\documentclass[11pt]{amsart}
\usepackage{amsmath}
\usepackage{amssymb}
\usepackage{amscd}
\def\NZQ{\mathbb}               % the font for N,Z,Q,R,C

\def\QQ{{\NZQ Q}}
\def\ZZ{{\NZQ Z}}
\def\RR{{\NZQ R}}

\def\PP{{\NZQ P}}

\newtheorem{Theorem}{Theorem}[section]
\newtheorem{Lemma}[Theorem]{Lemma}
\newtheorem{Corollary}[Theorem]{Corollary}
\newtheorem{Proposition}[Theorem]{Proposition}
\newtheorem{Remark}[Theorem]{Remark}

\newtheorem{Definition}[Theorem]{Definition}

\let\epsilon\varepsilon
\let\phi=\varphi
\let\kappa=\varkappa

\begin{document}

\title{Teissier's problem on inequalities of nef divisors over an arbitrary field}
\author{Steven Dale Cutkosky}
\thanks{Partially supported by NSF}

\address{Steven Dale Cutkosky, Department of Mathematics,
University of Missouri, Columbia, MO 65211, USA}
\email{cutkoskys@missouri.edu}

\begin{abstract}
  Boucksom, Favre and Jonsson establish in  \cite{BFJ} an analog of Diskant's inequality
in convex geometry for nef and big line bundles on a complete algebraic variety over an algebraically closed field of characteristic zero (Theorem F \cite{BFJ}), from which they deduce a  solution (Theorem D \cite{BFJ}) to a problem of Teissier (Problem B \cite{T}) on proportionality of nef divisors.  
In this paper we extend   these results to a complete variety over an arbitrary field $k$. 
\end{abstract}

\maketitle

\section{Introduction} 
In their beautiful paper \cite{BFJ},  Boucksom, Favre and Jonsson establish an analog of Diskant's inequality
in convex geometry for nef and big line bundles on a complete algebraic variety over an algebraically closed field of characteristic zero (Theorem F \cite{BFJ}), from which they deduce a  solution (Theorem D \cite{BFJ}) to a problem of Teissier (Problem B \cite{T}) on proportionality of nef divisors.   Teissier shows in \cite{T} that  a version of Bonnesen's inequality holds for a class of nef line bundles on a complete surface. Diskant's original inequality \cite{Di} is developed as a higher dimensional generalization of Bonnesen's inequality \cite{Bo}, and Boucksom, Favre and Jonsson establish this formula for nef and big line bundles on a complete variety of any dimension, over an algebraically closed field of arbitrary characteristic. This inequality is sufficiently strong to establish (in Theorem D \cite{BFJ}) that equality in the Khovanskii-Teissier inequalities (Corollary \ref{KT2}) for nef and big line bundles (on a complete variety over an algebraically closed field of characteristic zero) holds if and only if 
the line bundles are numerically proportional.

The purpose of  this paper is extend   these results to a complete variety over an arbitrary field $k$.

To obtain their Diskant inequality, Boucksom, Favre and Jonsson develop a theory of ``positive intersection products'', interpret the volume of a big line bundle as a positive intersection product (Theorem 3.1 \cite{BFJ}) and interpret the directional derivitive of the volume of a big divisor as a positive intersection product (Theorem A \cite{BFJ}). 
The positive intersection product is defined by  realizing the  product as a ``Weil Class'' in $N^p(\mathcal X)$. Here $\mathcal X$ is the Zariski Riemann manifold associated to a complete $d$-dimensional variety $X$. They realize $N^p(\mathcal X)$ for $0\le p\le d$ as an inverse limit of the finite dimensional real vector spaces $N^p(Y)$ where $Y$ is a nonsingular projective variety which birationally dominates $X$ by a morphism, and $N^p(Y)$ is the real vector space of numerical equivalence classes of codimension $p$-cycles on $Y$. 
The authors refer to Chapter 19 of \cite{F}, where the theory of numerical equivalence on nonsingular varieties is surveyed. 
$N^p(\mathcal X)$ is given the inverse limit topology (weak topology). They also develop a theory of ``Cartier classes'' on $\mathcal X$, by computing the direct limit $CN^p(\mathcal X)$ of the $N^p(Y)$, and giving them the direct limit topology (strong topology).
The  idea of the positive intersection product  of Cartier classes $\alpha_1,\ldots,\alpha_p \in CN^1(\mathcal X)$ is to take the limit
over all $Y$ of the ordinary intersection products $\beta_1\cdot\ldots\cdot\beta_p$ where $\beta_1,\ldots,\beta_p$ are Fujita approximations of $\alpha_1,\ldots,\alpha_p$ on $Y$. They develop the theory of these intersection products, and use this to prove the
Theorems A, D and F mentioned above.

We use the notation on schemes and varieties from Hartshorne \cite{H}. In particular,  a complete variety over a field $k$ is an integral $k$-scheme which is proper over $k$.

  We now discuss what the obstacles are to
extending the results of 
\cite{BFJ} to a complete variety over an arbitrary field.
The most daunting problem is that resolution of singularities is not known to be true for varieties of dimension larger than three over a field 
of positive characteristic. As such, we cannot use the sophisticated intersection theory from Chapter 19 \cite{F}. That theory  requires that the variety be smooth over an algebraically closed ground field. Certainly the assumption that $Y$ is smooth is necessary here.

The theory of numerical equivalence for line bundles has been developed for a proper scheme over an algebraically closed field in Kleiman's paper \cite{K}. The basic intersection theory here originates from an approach of Snapper \cite{Sn}, and is remarkably simple, so it is valid in a very high level of generality. A study of the paper \cite{K} shows that the  theory that we need for numerical equivalence of line bundles extends without difficulty to  an arbitrary field. We discuss this in Subsection \ref{SSInt}.

The volume of a line bundle $\mathcal L$ on a complete variety is defined as a lim sup,
 $$
\mbox{vol}_X(\mathcal L)=\limsup_{m\rightarrow\infty}\frac{\dim_k\Gamma(X,\mathcal L^m)}{m^d/d!}.
$$
This lim sup is actually a limit. When $k$ is an algebraically closed field of characteristic zero, this is shown in Example 11.4.7 \cite{L}, as a consequence of   Fujita Approximation \cite{F2} (c.f. Theorem 10.35 \cite{L}). The limit is established in 
 \cite{LM} and  \cite{Ta} when $k$ is algebraically closed of arbitrary characteristic. A proof over an arbitrary field is given in \cite{C}. In this paper, we deduce, in Theorem \ref{FApp},  Fujita Approximation over an arbitrary field from 
Theorem 3.3  \cite{LM} and Theorem 7.2 \cite{C}. We will need this to obtain the main results of this paper.

It is worth remarking that
when $k$ is an arbitrary field and $X$ is geometrically integral over $k$ we easily obtain that volume is a  limit by making the base change to
$\overline X=X\times_k\overline k$ where $\overline k$ is an algebraic closure of $k$. Then the volume of $\mathcal L$ (on $X$ over $k$) is equal to the volume of $\overline{\mathcal L}=\mathcal L\otimes_k\overline k$ (on $\overline X$ over $\overline k$). $\overline X$
is a complete $\overline k$ variety (it is integral) since $X$ is geometrically integral. Thus the conclusions of \cite{L}, \cite{LM} and \cite{Ta} are valid for $\overline{\mathcal L}$ (on the complete variety $\overline X$ over the algebraically closed field $\overline k$)
so that the  volume is a limit for $\mathcal L$ (on $X$ over $k$) as well. However, this argument is not applicable when $X$ is not geometrically integral. The most dramatic difficulty can occur when $k$ is not perfect, as there exist simple examples of irreducible
projective varieties which are not even generically reduced after taking the base change to the algebraic closure (we give a simple example below). 
In Example 6.3 \cite{C1} and Theorem 9.6 \cite{C} it is shown that for general graded linear series the limit does not always exist if 
$X$ is not generically reduced.

We now give an example, showing that even if $X$ is normal and $k$ is algebraically closed in the function field of $X$, then 
$X\times_k\overline k$ may not be generically reduced, where $\overline k$ is an algebraic closure of $k$.
Let $p$ be a prime number, $F_p$ be the field with $p$ elements and let $k=F_p(s,t,u)$ be a rational function field in three variables over $F_p$. Let $R$ be the local ring $R=(k[x,y,z]/(sx^p+ty^p+uz^p))_{(x,y,z)}$ with maximal ideal $m_R$.  $R$ is the localization of
$T=F_p[s,t,u,x,y,z]/(sx^p+ty^p+uz^p)$ at the ideal $(x,y,z)$, since $F_p[s,t,u]\cap(x,y,z)=(0)$. $T$ is nonsingular in codimension 1 by the Jacobian criterion over the perfect field $F_p$, and so
$T$ is normal by Serre's criterion. Thus $R$ is normal since it is a localization of $T$. Let $k'$ be the algebraic closure of $k$ in the quotient field $K$ of $R$. Then $k'\subset R$ since $R$ is normal. $R/m_R\cong k$ necessarily contains $k'$, so $k=k'$. However, we have that $R\otimes_k\overline k$ is generically not reduced, if $\overline k$ is an algebraically closure of $k$. Now taking $X$ to be a normal projective model of $K$ over $k$ such that $R$ is the local ring of a closed point of $X$, we have the desired example.
In fact, we have that $k$ is algebraically closed in $K$, but $K\otimes_k\overline k$ has nonzero nilpotent elements.

Lazarsfeld has shown that  the function ${\rm vol}_X$ on line bundles extends uniquely to a continuous  function on $N^1(X)$ which is homogeneous of degree 1. The proof in  Corollary 2.2.45 \cite{L} extends to the case of an arbitrary field.

In Theorem A of \cite{BFJ} it is proven that volume is continuously differentiable on the big cone of $X$, when the ground field $k$
is algebraically closed of characteristic zero. It is proven by Lazarsfeld and Mustata when the ground field $k$ is algebraically closed of arbitrary characteristic
in Remark 2.4.7 \cite{LM}. We establish that the volume is continuously differentiable when $X$ is a complete variety over an aribtrary field in Theorem \ref{TheoremA}.

In Example 2.7 of the survey  \cite{ELMNP1} by Ein, Lazarsfeld, Mustata, Nakamaye and Popa, it is shown that volume is not twice differentiable on the big cone of the blow up of $\PP^2$ at a $k$-rational point.

To define the positive intersection product over an arbitrary field we only need the intersection theory developed in \cite{K}.
We consider the directed system $I(X)$ of projective varieties $Y$ which have a birational morphism to $X$. On each $Y$ we consider
for $0\le p\le d=\dim X$ the finite dimensional real vector space 
$L^p(Y)$ of $p$-multilinear forms on $N^1(Y)$. We give $L^p(Y)$ the Euclidean topology, and take the inverse limit over $I(Y)$ 
$$
L^p(\mathcal X)=\lim_{\leftarrow}L^p(Y)
$$
and give it the strong topology (the inverse limit topology). $L^p(\mathcal X)$ is then a Hausdorff topological vector space.
We define the pseudo effective cone $\mbox{Psef}(L^p(Y))$ in $L^p(Y)$ to be the Zariski closure of the cone generated by the natural image of the $p$-dimensional closed subvarieties of $Y$. When $p=0$, we just take $L^0(Y)$ to be the real numbers, and the pseudoeffective cone to be the nonnegative real numbers. The inverse limit of the $\mbox{Psef}(L^p(Y))$ is then a closed convex and strict cone $\mbox{Psef}(L^p(\mathcal X))$ in $L^p(\mathcal X)$, allowing us to define a partial order $\ge$ on $L^p(\mathcal X)$.
In the case when $p=0$, we have that $L^0(\mathcal X)$ is the real numbers, and $\ge$ is just the usual order.

In Section \ref{SecPos}. we generalize the definition of the positive intersection product in \cite{BFJ}, essentially by defining the positive intersection product of $p$ big classes $\alpha_1,\ldots,\alpha_p$ in $L^{d-p}(\mathcal X)$
 as a limit of the intersection products by nef divisors $\beta_1,\ldots,\beta_{p}$ such that $\beta_i\le \alpha_i$ for all $i$ (where $\beta_i$ is in some $N^1(Y)$). This product can be considered as a multilinear form. We then show in the remainder of Section \ref{SecPos} that the properties of the positive intersection product which are obtained in \cite{BFJ} hold for our more general construction. Finally, we show that the proofs of the main theorems Theorem A, Theorem D and Theorem F of \cite{BFJ} extend with our
generalization of the positive intersection product. We make use of the ingenious manipulation of inequalities from their paper.

In Section \ref{SecVolume} we deduce Theorem \ref{TheoremA}, which is proven in Theorem A \cite{BFJ} when $k$ is algebraically closed of characteristic zero. 

In the final section, Section \ref{SecIneq}, we discuss  inequalities for nef line bundles, including the wonderful formulas of  Khovanskii and Teissier.  We establish Discant's inequality over an arbitrary field $k$ in Theorem \ref{TheoremF}. As an immediate corollary, we  obtain the following theorem, Theorem \ref{TheoremH}, which is an extension  of Proposition 3.2 \cite{T} to all dimensions and to arbitrary nef and big line bundles.

\begin{Theorem}(Theorem \ref{TheoremH}) Suppose that $\alpha,\beta\in N^1(X)$ are nef with $(\alpha^d)>0$, $(\beta^d)>0$   on a complete $d$-dimensional variety $X$ over a field $k$.
Then
$$
\frac{s_{d-1}^{\frac{1}{d-1}}-(s_{d-1}^{\frac{d}{d-1}}-s_0^{\frac{1}{d-1}}s_d)^{\frac{1}{d}}}{s_0^{\frac{1}{d-1}}}
\le r(\alpha;\beta)\le \frac{s_d}{s_{d-1}}\le\frac{s_1}{s_0}\le R(\alpha;\beta)\le 
\frac{s_d^{\frac{1}{d-1}}}{s_1^{\frac{1}{d-1}}-(s_1^{\frac{1}{d-1}}-s_d^{\frac{1}{d-1}}s_0)^{\frac{1}{d}}}
$$
where $s_i=(\alpha^i\cdot\beta^{d-i})$,
$$
r(\alpha;\beta)=\sup\{t\mid \alpha-t\beta\mbox{ is pseudo effective}\}\mbox{ is the ``inradius''}
$$
and
$$
R(\alpha;\beta)=\inf\{t\mid t\alpha-\beta\mbox{ is pseudo effective}\}\mbox{ is the ``outradius''}.
$$
\end{Theorem}

As a consequence, we deduce in Theorem \ref{TheoremA} that equality of the Khovanskii-Teissier inequalities is equivalent to 
$\alpha$ and $\beta$ being proportional in $N^1(X)$, extending the result in Theorem A \cite{BFJ} to arbitrary fields.

\section{Preliminaries on intersection theory and associated cones}\label{SecPre}

\subsection{Intersection theory on schemes}\label{SSInt} In this subsection we suppose that $X$ is a $d$-dimensional proper scheme over a  field $k$.
We begin by recalling some results from  Kleiman's paper \cite{K}, and their extension to arbitrary fields (some of the following is addressed in \cite{L2} and \cite{K2}).  Given a coherent sheaf
$\mathcal F$ on $X$ whose support has dimension $\le t$, and invertible sheaves $\mathcal L_1,\ldots, \mathcal L_t$ on $X$, there is an intersection product
$$
(\mathcal L_1\cdot\ldots\cdot \mathcal L_t\cdot \mathcal F)
$$
on $X$. 
The Euler characteristic $\chi(\mathcal N)$ of a coherent sheaf $\mathcal N$ of $\mathcal O_X$-modules is defined as
$$
\chi_k(\mathcal N)=\sum_{i\ge 0}(-1)^i\dim_kH^i(X,\mathcal N).
$$
Let $\mathcal L_1,\ldots,\mathcal L_t$ be $t$ invertible sheaves on $X$. Then
$$
\chi_k(\mathcal F\otimes\mathcal L^{n_1}\otimes\cdots\otimes\mathcal L_t^{n_t})
$$
is a numerical polynomial (Snapper \cite{Sn}, page 295 \cite{K}). The intersection number
$$
(\mathcal L_1\cdot\ldots\cdot \mathcal L_t\cdot\mathcal F)
$$
is defined to be the coefficient of the monomial $n_1,\ldots,n_t$ in $\chi_k(\mathcal F\otimes\mathcal L^{n_1}\otimes\cdots\otimes\mathcal L_t^{n_t})$. This number depends on the  ground field $k$.

This product is characterized by the nice properties established in Chapter 1, Section 2 \cite{K}. 
We will write 
$$
(\mathcal L_1\cdot\ldots\cdot\mathcal L_d)= (\mathcal L_1\cdot\ldots\cdot\mathcal L_d\cdot X)=(\mathcal L_1\cdot\ldots\cdot\mathcal L_d\cdot\mathcal O_X).
$$

The following lemma uses the notation of Proposition I.2.4 \cite{K}.

\begin{Lemma}\label{PA} Suppose that $X$ is a projective scheme over a field $k$, $H$ is an ample Cartier divisor on $X$ and $\mathcal F$ is a nonzero coherent $\mathcal O_X$-module. Then there exists $s\in \ZZ_{>0}$ and $\Delta$  in the complete linear system $|\mathcal O_X(sH)|$ such that $\Delta\cap {\rm Ass}(\mathcal F)=\emptyset$.
\end{Lemma}

\begin{proof} There exists $r\in \ZZ_{>0}$ such that $rH$ is very ample. Thus $X=\mbox{Proj}(S)$ where $S=\bigoplus_{n\ge 0}S_n$ is a standard graded, saturated $k$-algebra with $\mathcal O_X(rH)\cong \mathcal O_X(1)$. $\mathcal F\cong \tilde M$ is the sheafification of a finitely generated graded $S$-module $M$ such that $S_+=\bigoplus_{n>0}S_n$ is not an associated prime of $M$. By graded prime avoidance (c.f. Lemma 1.5.10 \cite{BH}) there exists a homogeneous element $h\in S_n$ for some $n$ such that $h\not\in P$ for any prime ideal
$P\in {\rm Ass}(M)$.  If $S_1$ is spanned by $x_0,\ldots,x_t$ as a $k$-vector space, then the effective Cartier divisor
$$
\Delta=\{(\mbox{Spec}(S_{(x_i)}),\frac{h}{x_i^n})\mid 0\le i\le t\}
$$
is linearly equivalent to $nrH$ and $\Delta\cap {\rm Ass}(\mathcal F)=\emptyset$.
\end{proof}

The following version of Bertini's theorem will be useful. The theorem follows from Theorem 3.4.10 and Corollary 3.4.14 \cite{FOV}.

\begin{Theorem}\label{Bertini}(Bertini's Theorem) Suppose that $X$ is a projective variety over an infinite field $k$ and $A$ is a very ample (integral) divisor on $X$. Then Bertini's theorems are valid for a generic member of $|\mathcal O_X(A)|$.
\end{Theorem}

There exists a nontrivial Zariski open subset $U$ of a projective space $\PP^n_k$ parametrizing the complete linear system $|\mathcal O_X(A)|$, over which the desired Bertini conditions hold. With the assumption that $k$ is infinite, $U$ contains infinitely many $k$-rational points (c.f. Theorem 2.19 \cite{J}). The corresponding elements of $|\mathcal O_X(A)|$ are called ``generic members'' in \cite{FOV}.

$\mathcal L\in \mbox{Pic}(X)$ is said to be numerically equivalent to zero, written $\mathcal L\equiv 0$, if $(\mathcal L\cdot C)=0$ for all closed integral curves $C\subset X$. The intersection product is defined modulo numerical  equivalence.
Let $N^1(X)$ be the real vector space $(\mbox{Pic}(X)/\equiv)\otimes_{\ZZ}\RR$.

The following proposition extends to the case of an arbitrary field a classical theorem.

\begin{Proposition}\label{Picard} Suppose that $X$ is a proper scheme over a field $k$. Then $N^1(X)$ is a finite dimensional real vector space.
\end{Proposition}

\begin{proof} In the case when $k$ is algebraically closed, this is proven in Proposition IV.1.4 \cite{K}.

Let $\overline k$ be an algebraic closure of $k$. Let $C\subset X\otimes_k\overline k$ be a closed integral curve. Let $\{U_1,\ldots,U_n\}$ be an affine cover of $X$. Then $\{U_1\otimes_k\overline k,\ldots,U_n\times_k\overline k\}$ is an affine cover of $X\times_k\overline k$. There exists a finite extension field $L$ of $k$ (a field of definition of $C$) such that $\mathcal I_C|(U_i\otimes_k\overline k)$ is defined over $L$ for all $i$ (a set of generators of $\Gamma(U_i\times_k\overline k,\mathcal I_C)$
is contained in  $\Gamma(U_i,\mathcal O_X)\otimes_k L$ for all $i$). Let $C'\subset X\times_kL$ be this integral curve, so that $C=C'\times_L\overline k$. Let $g:X\times_kL\rightarrow X$ be the natural proper morphism, and let $\gamma=g(C')$.
Let $f:X\times_k\overline k\rightarrow X$ be the natural morphism. 

Suppose that $\mathcal L$ is a line bundle on $X$. We compute $(f^*\mathcal L\cdot C)$, taking $\overline k$ as the ground field of
$X\times_k\overline k$. By flat base change of cohomology (Proposition III.9.3 \cite{H}),
$$
\chi_L(g^*\mathcal L^n\otimes \mathcal O_{C'})=\chi_{\overline k}(f^*\mathcal L^n\otimes\mathcal O_C).
$$
Thus
$$
[L:k](f^*\mathcal L\cdot C)=(g^*\mathcal L\cdot C'),
$$
where $X\times_kL$ is regarded as a scheme over $k$. By Proposition I.2.6 \cite{K},
$$
(g^*\mathcal L\cdot C')=[k(C'):k(\gamma)](\mathcal L\cdot \gamma),
$$
where $k(\gamma)$ and $k(C')$ are the respective functions fields of $\gamma$ and $C'$.
Thus
\begin{equation}\label{eq120}
(f^*\mathcal L\cdot C)=\frac{[k(C'):k(\gamma)]}{[L:k]}(\mathcal L\cdot \gamma).
\end{equation}

Suppose that $\gamma$ is a closed integral curve on $X$. Apply (\ref{eq120}) to any of the finitely many integral curves $C$ such that $C\subset \gamma\times_k\overline k$ to obtain that for $\mathcal L\in \mbox{Pic}(X)$, $f^*\mathcal L\equiv 0$ on $X\times _k\overline k$ if and only if $\mathcal L\equiv 0$ on $X$. Thus there is a well defined inclusion of real vector spaces
$$
f^*:N^1(X)\rightarrow N^1(X\times_k\overline k),
$$
so that 
$$
\dim_{\RR}N^1(X)\le \dim_{\RR}N^1(X\times_k\overline k)<\infty.
$$
\end{proof}

We give $N^1(X)$ the Euclidean topology.

\subsection{$N^1(X)$ on a variety}
In this subsection we  suppose that $X$ is a $d$-dimensional complete  variety over a field $k$. Then $\mbox{Pic}(X)$ is isomorphic to the group of Cartier divisors $\mbox{Div}(X)$ on $X$.  The basic theory of $N^1(X)$ (written as $N^1(X)_{\RR}$) is developed in the first chapter of \cite{L}, in terms of Cartier divisors. We summarize a few of the concepts, which are valid over an arbitrary  base field  $k$. $D\in N^1(X)$ is $\QQ$-Cartier if $D$ is represented in $N^1(X)$ by a sum $D=\sum a_iE_i$ with $a_i\in \QQ$ and $E_i$ integral divisors.
An $\RR$-divisor $D$ is effective if it is represented by a sum $D=\sum a_iE_i$ where the $E_i$ are effective integral divisors and $a_i\in \RR_{\ge 0}$. It is ample if $D=\sum a_iA_i$ with 
$a_i\in \RR_{> 0}$ and $A_i$ ample integral divisors. It is nef (numerically effective) if $(D\cdot C)\ge 0$ for all closed integral curves $C$ on $X$. 

The ample cone $\mbox{Amp}(X)$ of $X$ is the convex cone in $N^1(X)$ of ample $\RR$-divisors, and the nef cone $\mbox{Nef}(X)$ is the convex cone in $N^1(X)$ of nef $\RR$-divisors. By Section 4 of \cite{K} and as exposed in Theorem 1.4.23 \cite{L}, we have  that $\mbox{Nef}(X)$ is closed and $\mbox{Amp}(X)$ is the interior of $\mbox{Nef}(X)$.

In Chapter 2 of \cite{L}, big and pseudoeffective cones are defined. 
An $\RR$-divisor $D$ on $X$ is big if $D=\sum a_iE_i$ where the $E_i$ are big integral divisors and $a_i\in \RR_{\ge 0}$.
The big cone $\mbox{Big}(X)$ is the convex cone in $N^1(X)$ of big $\RR$-divisors on $X$. The pseudoeffective cone $\mbox{Psef}(X)$ is the closure of the convex cone spanned by the classes of all effective $\RR$-divisors. In Theorem 2.2.26 \cite{L}, it is shown that
when $X$ is projective, $\mbox{Psef}(X)$ is the closure of $\mbox{Big}(X)$ and $\mbox{Big}(X)$ is the interior of $\mbox{Psef}(X)$.

Suppose that $X$ is projective. Since $\mbox{Amp}(X)$ is open, if $\alpha\in \mbox{Big}(X)$, then $\alpha$ has a representative
\begin{equation}\label{eq80}
\alpha=H+E
\end{equation}
in $N^1(X)$ where $H$ is an ample $\RR$-divisor and $E$ is an effective $\QQ$-divisor.

\subsection{partial orders on vector spaces}\label{SecPO}
Let $V$ be a  vector space and $C\subset V$ be a pointed (containing the origin) convex cone which is strict
($C\cap(-C)=\{0\}$). Then we have a partial order on $V$ defined by $x\le y$ if $y-x\in C$.

\subsection{Volume on a variety} The volume of a line bundle $\mathcal L$ on a $d$-dimensional complete variety over a field $k$ is 
$$
\mbox{vol}(\mathcal L)=\mbox{vol}_X(\mathcal L)=\limsup_{m\rightarrow\infty}\frac{\dim_k\Gamma(X,\mathcal L^m)}{m^d/d!}.
$$
This lim sup is actually a limit (shown in Example 11.4.7 \cite{L}  when the ground field $k$ is algebraically closed of characteristic zero, and in \cite{LM} or 
\cite{Ta} when $k$ is algebraically closed of arbitrary characteristic. A proof over an arbitrary field is given in \cite{C}).

Some important properties of volume are that $D\equiv D'$ implies $\mbox{vol}(D)=\mbox{vol}(D')$ (Proposition 2.2.41 \cite{L},  and
that for any positive integer $a$, 
\begin{equation}\label{eq85}
\mbox{vol}(aD)=a^d\mbox{vol}(D).
\end{equation} 
This identity is proven in Proposition 2.2.30 \cite{L}.
 
\begin{Theorem}(Lazarsfeld)\label{Volcont} Suppose that $X$ is a complete variety over a field $k$. Then the function ${\rm vol}_X$ on line bundles extends uniquely to a continuous  function on $N^1(X)$ which is homogeneous of degree 1.
\end{Theorem}

\begin{proof} This is proven in Corollary 2.2.45 \cite{L} when $X$ is projective and $k$ is algebraically closed of characteristic zero.
The proof extends without difficulty to the more general situation of this theorem, as we now indicate.

We first establish the theorem in the case when $X$ is projective over an infinite field $k$. The proof of Corollary 2.2.45 in \cite{L} extends to this case, after we make a couple of observations. First, Lemma 2.2.37 \cite{L} is valid when $k$ is an infinite field, since the $k$-rational points are then dense in a $\PP^n_k$ parameterizing a linear system of divisors. Second, by Theorem \ref{Bertini}, we can find the very general hyperplane sections required for the proof of Corollary 2.2.45 \cite{L}.

Now suppose that $X$ is projective over a finite field $k$. Let $k'=k(t)$ be a rational function field over $k$, and $X'=X\times_kk'$
with natural morphism $f:X'\rightarrow X$. $X'$ is a $k'$-variety. By the proof of Proposition \ref{Picard}, and by flat base change of cohomology
(Proposition III.9.3 \cite{H}), we have a commutative diagram
$$
\begin{array}{lll}
&&\RR\\
{\rm vol}_X&\nearrow&\uparrow {\rm vol}_{X'}\\
N^1(X)&\stackrel{f^*}{\rightarrow}&N^1(X').
\end{array}
$$
Thus $\mbox{vol}_X$ is continuous since $f^*$ and $\mbox{vol}_{X'}$ are.

The general case, when $X$ is complete over a field now follows from taking a Chow cover.
\end{proof}

\section{More vector spaces and cones associated to a variety}\label{SecMore}
In this section we suppose that $X$ is a complete $d$-dimensional variety over a  field $k$.

\subsection{Finite dimensional vector spaces and cones associated to a variety}
 For $0< p\le d$, we define
 $M^p(X)$ to be the direct product of $N^1(X)$ $p$ times, and we define $M^0(X)=\RR$.
 For $1< p\le d$,  we define
 $L^p(X)$ to be the vector space of $p$-multilinear forms from $M^p(X)$ to $\RR$, and define $L^0(X)=\RR$.

The intersection product gives us 
 $p$-multilinear maps
\begin{equation}\label{eq6}
M^p(X)\rightarrow L^{d-p}(X),
\end{equation}
 for $0\le p\le d$.
In the special case when $p=0$, the map is just the linear map taking $1$ to the map 
$$
(\mathcal L_1,\ldots,\mathcal L_d)\mapsto (\mathcal L_1\cdot\ldots\cdot \mathcal L_d)=(\mathcal L_1\cdot\ldots\cdot\mathcal L_d\cdot X).
$$
We will denote the image of $(\mathcal L_1,\ldots,\mathcal L_p)$ by $\mathcal L_1\cdot\ldots\cdot \mathcal L_p$.
We will sometimes write 
$$
\mathcal L_1\cdot\ldots\cdot\mathcal L_p(\beta_{p+1},\ldots,\beta_d)
=(\mathcal L_1\cdot\ldots\cdot\mathcal L_p\cdot\beta_{p+1}\cdot\ldots\cdot \beta_d).
$$

We give all of the above vector spaces the Euclidean topology, so that all of the mappings considered above are continuous.

 Let $|*|$ be a norm on $M^1(X)$ giving the Euclidean topology. The Euclidean topology on $L^p(X)$ is given by
the norm $||A||$, which is defined on a multilinear form $A\in L^p(X)$ to be the greatest lower bound of all real numbers $c$ such that
$$
|A(x_1,\ldots,x_p)|\le c|x_1|\cdots |x_p|
$$
for $x_1,\ldots,x_p\in M^1(X)$.

Suppose that $V$ is a closed $p$-dimensional  subvariety of $X$ with $1\le p\le d$. Define $\sigma_V\in L^p(X)$ by 
$$
\sigma_V(\mathcal L_1,\ldots,\mathcal L_p) = (\mathcal L_1\cdot\ldots\cdot \mathcal L_p\cdot V)
$$
for $\mathcal L_1,\ldots,\mathcal L_p\in \mbox{Pic}(X)$. The pseudoeffective cone $\mbox{Psef}(L^p(X))$ in $L^p(X)$ is the closure of the cone generated by the  $\sigma_V$ in $L^p(X)$.  
$\mbox{Psef}(L^0(X))$ is defined to be the nonnegative real numbers.

\begin{Lemma}\label{Lemma60}  Suppose that $X$ is a projective variety over a field $k$ and $1\le p\le d$.
\begin{enumerate}
\item[1)] Suppose that $\alpha\in \mbox{Psef}(L^p(X))$ and $\mathcal L_1,\ldots,\mathcal L_p\in M^1(X)$ are nef. Then
$$
\alpha(\mathcal L_1,\ldots,\mathcal L_p)\ge 0.
$$
\item[2)] $\mbox{Psef}(L^p(X))$ is a strict cone.
\end{enumerate}
\end{Lemma}

\begin{proof} First suppose that $\alpha=\sum a_i\sigma_{V_i}$
with $a_i\in \RR_{\ge 0}$ and $V_i$ closed $p$-dimension  subvarieties of $X$. Then if $\mathcal L_1,\ldots,\mathcal L_p$ are nef, we have that
$\alpha(\mathcal L_1,\ldots,\mathcal L_p)\ge 0$ (since the nef cone is the closure of the ample cone). Now suppose that $\alpha$ is an arbitrary element of $\mbox{Psef}(L^p(X))$.
Then $\alpha$ is the limit of a sequence $\alpha_j=\sum_ia_i^j\sigma_{V_i}$ with $a_i^j\in \RR_{\ge 0}$ and $V_i$  closed $p$-dimensional subvarieties of $X$. Suppose that $\mathcal L_1,\ldots,\mathcal L_p$ are nef. We will show that
$\alpha(\mathcal L_1,\ldots,\mathcal L_p)\ge 0$. Suppose otherwise; then $z:=\alpha(\mathcal L_1,\ldots,\mathcal L_p)<0$. 
Let $\epsilon=\frac{|z|}{|\mathcal L_1|\cdots|\mathcal L_p|}$. There exists $\alpha_j$ such that $||\alpha-\alpha_j||<\epsilon$
(where $||*||$ is the norm defined above). Thus 
$$
|\alpha(\mathcal L_1,\ldots,\mathcal L_p)-\alpha_j(\mathcal L_1,\ldots,\mathcal L_p)|<\epsilon|\mathcal L_1|\cdots|\mathcal L_p|=|z|,
$$
a contradiction since $\alpha_j(\mathcal L_1,\ldots,\mathcal L_p)\ge 0$.

In particular, if $\alpha\in \mbox{Psef}(L^p(X))\cap (-\mbox{Psef}(L^p(X)))$, then $\alpha$ vanishes on the open subset $\mbox{Amp}(X)^p$ 
of $M^p(X)$. Since $\mbox{Amp(X)}$ contains a basis of $M^1(X)$ and $\alpha$ is multilinear we have that $\alpha=0$.

\end{proof}

Since $\mbox{Psef}(L^p(X))$ is a strict cone, we  have by Section \ref{SecPO} an equivalence relation on $\mbox{Psef}(L^p(X))$, defined by 
$$
\alpha\ge 0\mbox{ if }\alpha\in \mbox{Psef}(L^p(X)).
$$

$L^0(X)=\RR$ and $\mbox{Psef}(L^0(X))$ is the set of nonnegative real numbers, so $\ge$ is the usual order on $\RR$.

\begin{Lemma}\label{Lemma7}
Suppose that $X$ is a projective variety over a field $k$ and $\beta\in \mbox{Psef}(L^p(X))$. Then  the set
$$
\{\alpha\in \mbox{Psef}(L^p(X))\mid 0\le \alpha\le\beta\}
$$
is compact. 
\end{Lemma}

\begin{proof}
 This set is equal to the closed set
$$
\mbox{Psef}(L^p(X))\cap (\beta-\mbox{Psef}(L^p(X))).
$$
Since $\mbox{Psef}(L^p(X))$ is strict, there exists an open set of  linear hyperplanes which intersect $\mbox{Psef}(L^p(X))$ only at the origin.
The choice of a linearly independent set of such hyperplanes and their translations by $\beta$ realizes a bounded set containing $\{\alpha\in \mbox{Psef}(L^p(X))\mid 0\le \alpha\le\beta\}$.
\end{proof}

Suppose that $Y$ is a complete variety over $k$ and $f:Y\rightarrow X$ is a birational morphism. Then  the natural homomorphism $f^*:\mbox{Pic}(X)\rightarrow \mbox{Pic}(Y)$
induces continuous linear maps
$f^*:N^1(X)\rightarrow N^1(Y)$ and $f_*:L^p(Y)\rightarrow L^p(X)$. By Proposition I.2.6 \cite{K}, for $1\le t\le d$, we have that
\begin{equation}\label{eq50}
f^*(\mathcal L_1)\cdot\ldots\cdot f^*(\mathcal L_t)=\mathcal L_1\cdot \ldots \cdot \mathcal L_t
\end{equation}
for $\mathcal L_1,\ldots,\mathcal L_t\in \mbox{Pic}(X)$. Thus for $0\le p\le d$ we have commutative diagrams of linear maps
\begin{equation}\label{eq1}
\begin{array}{rcl}
M^p(Y)&\rightarrow &L^{d-p}(Y) \\
f^*\uparrow&&f_*\downarrow \\
M^p(X)&\rightarrow &L^{d-p}(X),
\end{array}
\end{equation}
where the horizontal maps are those of (\ref{eq6}).

For $\alpha\in N^1(X)$, we have that
\begin{equation}\label{eq91}
f^*(\alpha)\in \mbox{Nef}(Y)\mbox{ if and only if }\alpha\in \mbox{Nef}(X).
\end{equation}
\begin{equation}\label{eq92}
f^*(\alpha)\in\mbox{Big}(Y)\mbox{ if and only if }\alpha\in \mbox{Big}(X).
\end{equation}
\begin{equation}\label{eq93}
f^*(\alpha)\in\mbox{Psef}(Y)\mbox{ if }\alpha\in \mbox{Psef}(X).
\end{equation}

\begin{Lemma}\label{Lemma1} Suppose that $Y$ is a complete variety over $k$ and $f:Y\rightarrow X$ is a birational morphism. Then $f_*(\mbox{Psef}(L^p(Y)))\subset \mbox{Psef}(L^p(X))$.
\end{Lemma}

\begin{proof} Suppose that $V$ is a $p$-dimensional closed subvariety of $Y$. Let $W=f(V)$. Then $f_*(\sigma_V)=\deg(f|V)\sigma_W$ if $\dim W=p$ and is zero otherwise, by Proposition I.2.6 \cite{K}. Since $f_*$ is continuous and $\mbox{Psef}(L^p(X))$ is closed, we have that
$$
f_*(\mbox{Psef}(L^p(Y))\subset \mbox{Psef}(L^p(X)).
$$
\end{proof}

\subsection{Infinite dimensional topological spaces associated to  a variety}

Let $I(X)=\{Y_i\}$ be the set of projective varieties whose function field is $k(X)$ and such that the birational map $Y_i\dasharrow X$
is a morphism. This makes $I(X)$ a directed set. 
$\{M^p(Y_i)\}$ is a directed system of real vector spaces, where we have a linear mapping 
$f_{ij}^*:M^p(Y_i)\rightarrow M^p(Y_j)$ if the birational map $f_{ij}:Y_j\rightarrow Y_i$ is a morphism. 
We define 
$$
M^p(\mathcal X)=\lim_{\rightarrow}M^p(Y_i)
$$
with the strong topology  (the direct limit topology, c.f. Appendix 1. Section 1 \cite{D}). $M^p(\mathcal X)$ is a real vector space. As a vector space, $M^p(\mathcal X)$ is isomorphic to the $p$-fold product $M^1(\mathcal X)^p$.

We define $\alpha\in M^1(\mathcal X)$ to be $\QQ$-Cartier (respectively nef, big, effective, pseudoeffective) if there exists a representative of $\alpha$ in $M^1(Y)$
which has this  property for some $Y\in I(X)$.  We define subsets $\mbox{Nef}^p(\mathcal X)$, $\mbox{Big}^p(\mathcal X)$ and 
$\mbox{Psef}^p(\mathcal X)$ to be the respective subsets of $M^p(\mathcal X)$ of nef, big and pseudoeffective divisors.  There are all convex cones in the vector space $M^p(\mathcal X)$.

By (\ref{eq91}), (\ref{eq92}) and (\ref{eq93}), $\{\mbox{Nef}(Y)^p\}$, $\{\mbox{Big}(Y)^p\}$ and $\{\mbox{Psef}(Y)^p\}$
also form directed systems. As sets, we have that
$$
 \mbox{Nef}^p(\mathcal X)=\lim_{\rightarrow}(\mbox{Nef}(Y)^p),\,\,
\mbox{Big}^p(\mathcal X)=\lim_{\rightarrow}(\mbox{Big}(Y)^p),\,\,
\mbox{Psef}^p(\mathcal X)=\lim_{\rightarrow}(\mbox{Psef}(Y)^p).
$$
We give all of these sets  their respective strong topologies. 

Let  $\rho_Y:M^p(Y)\rightarrow M^p(\mathcal X)$ be the induced continuous linear maps
for $Y\in I(X)$. We will also denote the induced continuous maps 
$\mbox{Nef}(Y)^p\rightarrow \mbox{Nef}^p(\mathcal X)$, $\mbox{Big}(Y)^p\rightarrow\mbox{Big}^p(\mathcal X)$
and $\mbox{Psef}(Y)^p\rightarrow\mbox{Psef}^p(\mathcal X)$ by $\rho_Y$.

$\{L^p(Y_i)\}$ is an inverse  system of topological vector spaces, where we have a linear map $(f_{ij})_*:L^p(Y_j)\rightarrow L^p(Y_i)$ if the birational map $f_{ij}:Y_j\rightarrow Y_i$ is a morphism. 
We  define  
$$
L^p(\mathcal X)=\lim_{\leftarrow}L^p(Y_i),
$$
with the weak topology (the inverse limit topology).

In general, good topological properties on a directed system do not extend to the direct limit (c.f. Section 1 of Appendix 2 \cite{D}, especially the remark before 1.8). However, good topological properties on an inverse system do extend (c.f. Section 2 of Appendix 2 \cite{D}). In particular, we have the following proposition.

\begin{Proposition}
 $L^p(\mathcal X)$ is a  Hausdorff real topological vector space.
$L^p(\mathcal X)$ is a real vector space which is isomorphic (as a vector space) to the $p$-multilinear forms on $M^1(\mathcal X)$.
\end{Proposition} 

Let $\pi_Y:L^p(\mathcal X)\rightarrow M^p(Y)$  be the induced continuous linear maps
for $Y\in I(X)$.

We will make repeated use of the following, which follow from the universal properties of the inverse limit and the direct limit
(c.f. Theorems 2.5 and 1.5 \cite{D}).

\begin{Lemma} Suppose that $\mathcal F$ is one of $M^p$, $\mbox{Nef}^p$, $\mbox{Big}^p$ or $\mbox{Psef}^p$.
Then  giving a continuous mapping
$$
\Phi:\mathcal F(\mathcal X)\rightarrow L^{d-p}(\mathcal X)
$$
is equivalent to giving continuous maps $\phi_Y:\mathcal F(Y)\rightarrow L^{d-p}(Y)$ for all $Y\in I(X)$, such that the diagram
$$
\begin{array}{lll}
\mathcal F(Z)&\stackrel{\phi_Z}{\rightarrow}&L^{d-p}(Z)\\
f^*\uparrow&&\downarrow f_*\\
\mathcal F(Y)&\stackrel{\phi_Y}{\rightarrow}&L^{d-p}(Y)
\end{array}
$$
commutes, whenever $f:Z\rightarrow Y$ is in $I(Y)$.
\end{Lemma}
In the case when $\mathcal F=M^p$, if the $\phi_Y$ are all multilinear, then $\Phi$ is also multilinear (via the vector space isomorphism
of $M^p(\mathcal X)$ with $p$-fold product $M^1(\mathcal X)^p$).

As an application, we have the following useful property.

\begin{Lemma}\label{Intcont} The  intersection product gives us a continuous 
map
$$
\mathcal F(\mathcal X)\rightarrow L^{d-p}(\mathcal X)
$$
whenever $\mathcal F$ is one of $M^p$, $\mbox{Nef}^p$, $\mbox{Big}^p$ or $\mbox{Psef}^p$.
The map is multilinear on $M^p(\mathcal X)$.
\end{Lemma}
We will denote the image of $(\alpha_1,\ldots,\alpha_p)$ by $\alpha_1\cdot\ldots\cdot\alpha_p$. 
For $\beta_{p+1},\ldots,\beta_d\in M^1(\mathcal X)$, we will often write
$$
\alpha_1\cdot\ldots\cdot\alpha_p(\beta_{p+1},\ldots,\beta_d)=(\alpha_1\cdot\ldots\cdot\alpha_p\cdot\beta_{p+1}\cdot\ldots\cdot\beta_d).
$$

\subsection{Pseudoeffective classes in $L^p(\mathcal X)$}

We define a class $\alpha\in L^p(\mathcal X)$ to be pseudoeffective if $\pi_Y(\alpha)\in L^p(Y)$ is pseudoeffective for all $Y\in I(X)$.

\begin{Lemma}\label{Lemma3}
The set of pseudoeffective classes $\mbox{Psef}(L^p(\mathcal X))$ in $L^p(\mathcal X)$
is a  strict closed convex cone in $L^p(\mathcal X)$. 
\end{Lemma}

\begin{proof} The fact that $\mbox{Psef}(L^p(\mathcal X))$ is a closed convex cone follows from the fact that
$$
\mbox{Psef}(L^p(\mathcal X))=\cap_{Y\in I(X)}\pi_Y^{-1}(\mbox{Psef}(L^p(Y)))
$$
 is an intersection of closed convex cones.
 To verify strictness, we must check that if $\alpha$ and $-\alpha\in \mbox{Psef}(L^p(\mathcal X))$, then $\alpha=0$.
 Suppose this is not the case. Then there exists a nonzero $\alpha$ such that $\alpha,-\alpha\in \mbox{Psef}(L^p(\mathcal X))$.
 Then $\pi_Y(\alpha)$, $-\pi_Y(\alpha)\in \mbox{Psef}(L^p(Y))$ for all $Y\in I(X)$ so that $\pi_Y(\alpha)=0$ for all $Y$ by Lemma (\ref{Lemma60}), and thus $\alpha=0$.
  \end{proof}
 
 By  Lemma \ref{Lemma3} (c.f. Section \ref{SecPO}), we can define a partial order $\ge 0$ on $L^p(\mathcal X)$ by $\alpha\ge 0$ if $\alpha\in \mbox{Psef}(L^p(\mathcal X))$.

$L^0(\mathcal X)=\RR$ and $\mbox{Psef}(L^0(\mathcal X))$ is the set of nonnegative real numbers (by the remark before Lemma \ref{Lemma7}), so $\ge$ is the usual order on $\RR$.

\begin{Lemma}\label{Lemma61}
Suppose that $\mathcal L_1,\ldots,\mathcal L_p\in \mbox{Nef}(\mathcal X)$ and $\alpha\in \mbox{Psef}(L^p(\mathcal X))$. Then
$$
\alpha(\mathcal L_1,\ldots,\mathcal L_p)\ge 0.
$$
\end{Lemma}

\begin{proof} Suppose that $Y\in I(X)$ is such that $\mathcal L_1,\ldots,\mathcal L_p$ are represented by classes in $M^1(Y)$. Then
$$
\alpha(\mathcal L_1,\ldots,\mathcal L_p)=\pi_Y(\alpha)(\mathcal L_1,\ldots,\mathcal L_p)\ge 0
$$
by Lemma \ref{Lemma60}, since $\pi_Y(\alpha)\in \mbox{Psef}(L^p(Y))$.
\end{proof}

 \begin{Lemma} Suppose that $V\subset Y$ is a $p$-dimensional closed subvariety of $Y$. Then there exists $\alpha\in \mbox{Psef}(L^p(\mathcal X))$ such that $\pi_Y(\alpha)=\sigma_V$.
 \end{Lemma}
 
 \begin{proof} By the existence theorem, Theorem 37 of Section 16, Chapter VI, page 106 \cite{ZS2}, there exists a rank 1 $p$-dimensional valuation  $\nu$ of $k(X)$ whose center on $X$ is $V$. For $Z\in I(X)$, let $V_Z$ be the center of $\nu$ on $Z$. Define
 $\alpha\in L^p(\mathcal X)$ by
 $$
 \pi_Z(\alpha)=\frac{[k(V_W):k(V_Z)]}{[k(V_W):k(V)]}\sigma_{V_Z}
 $$
 if $\dim V_Z=p$ and there exists a diagram in $I(X)$
 $$
 \begin{array}{lllll}
 &&W&&\\
 &\swarrow&&\searrow&\\
 Y&&&&Z.
 \end{array}
 $$
 Define $\pi_Z(\alpha)=0$ if $\dim V_Z<p$.
 \end{proof}

\begin{Lemma}\label{Lemma2} Suppose that $\alpha\in \mbox{Psef}(L^p(\mathcal X))$. Then the set
$$
\{\beta\in L^p(\mathcal X) | 0\le \beta \le \alpha\}
$$
is compact.
\end{Lemma}

\begin{proof} Let $K=\{\beta\in L^p(\mathcal X) | 0\le \beta \le \alpha\}$ and $K_Y=\{\beta\in L^p(Y) | 0\le \beta \le \pi_Y(\alpha)\}$
for $Y\in I(X)$. The statement that $\beta\in K$ is equivalent to the statement that  $\beta$ and $\alpha-\beta$ are in $\mbox{Psef}(L^p(\mathcal X))$, which is equivalent to the statement that  $\pi_Y(\beta)$ and $\pi_Y(\alpha)-\pi_Y(\beta)$ are in $\mbox{Psef}(L^p(Y))$ for all $Y\in I(X)$, which is the statement that $\pi_Y(\beta)\in K_Y$ for all $Y\in I(X)$.
Thus $K=\cap_Y\pi_Y^{-1}(K_Y)$. The $K_Y$ form an inverse system and $\lim_{\leftarrow}K_Y$ is homeomorphic to the subspace 
$\cap_Y\pi_Y^{-1}(K_Y)$ of $L^p(\mathcal X)$ (c.f. 2.8, Appendix 2 \cite{D}). Since the $K_Y$ are all compact by Lemma \ref{Lemma7},
$\lim_{\leftarrow} K_Y$ is compact (c.f. 2.4, Appendix 2 \cite{D}).
\end{proof}

\begin{Lemma}\label{Lemma4} Suppose that $\alpha_i\in M^1(\mathcal X)$ for $1\le i\le p$, with $\alpha_1$ psef and $\alpha_i$ nef for $i\ge 2$. Then $\alpha_1\cdot\ldots\cdot \alpha_p\in L^{d-p}(\mathcal X)$ is psef.
\end{Lemma}

\begin{proof} There exists $Y\in I(X)$  such that $\alpha_1,\ldots, \alpha_p$ are represented on $Y$ by classes $\mathcal N_1,\ldots,\mathcal N_p\in M^1(Y)$, with $\mathcal N_1$ psef and $\mathcal N_i$ nef for $i\ge 2$. We will show that 
$\mathcal N_1\cdot\ldots\cdot \mathcal N_p\in L^{d-p}(Y)$ is psef. Let $H$ be very  ample  on $Y$. We will show that
$$
(\mathcal N_1+tH)\cdot(\mathcal N_2+tH)\cdot \ldots \cdot (\mathcal N_p+tH)
$$
is psef for all $t> 0$. By continuity of the intersection product, and the fact that $\mbox{Psef}(L^{d-p}(Y))$ is closed in $M^1(Y)$,  we will conclude that $\mathcal N_1\cdot \ldots\cdot \mathcal N_p$ is psef.

For $2\le j\le p$, and since the closure of $\mbox{Amp}(Y)$ is $\mbox{Nef}(Y)$, we have expressions
$$
N_i+tH\equiv \sum a_{ij}H_{ij}
$$
with $H_{ij}$ integral ample divisors on $Y$ and $a_{ij}\in \RR_{\ge 0}$. Since the closure of $\mbox{Big}(Y)$ is $\mbox{Psef}(Y)$, we have an expression
$$
\mathcal N_1+tH\equiv \sum b_jD_j
$$
with $D_j$ integral divisors on $Y$ and $b_j\in \RR_{\ge 0}$. By multilinearity of the intersection product, it suffices to show that
each
$$
H_2\cdot \ldots\cdot H_p\cdot D
$$
is psef, where $H_i$ is any of the $H_{ij}$ and $D$ is any of the $D_j$.  Suppose that $\mathcal L_1,\ldots,\mathcal L_{d-p}\in M^1(Y)$.
Using Propositions I.2.4 and I.2.5 of \cite{K} and Lemma \ref{PA}, we compute
\begin{equation}\label{eq200}
\begin{array}{l}
(\mathcal L_1\cdot \ldots\cdot \mathcal L_{d-p}\cdot \mathcal O_Y(H_2)\cdot \ldots\cdot \mathcal O_Y(H_p)\cdot \mathcal O_Y(D))\\
=  (\mathcal L_1\cdot \ldots\cdot \mathcal L_{d-p}\cdot \mathcal O_Y(H_2)\cdot \ldots\cdot \mathcal O_Y(H_p)\cdot D)\\
=(\mathcal L_1\otimes \mathcal O_D\cdot \ldots\cdot\mathcal L_{d-p}\otimes \mathcal O_D\cdot \mathcal O_Y(H_1)\otimes \mathcal O_D\cdot \ldots\cdot \mathcal O_Y(H_p)\otimes \mathcal O_D)\\
=\frac{1}{s_p}(\mathcal L_1\otimes \mathcal O_D\cdot \ldots\cdot \mathcal L_{d-p}\otimes \mathcal O_D\cdot \mathcal O_Y(H_2)\otimes \mathcal O_D\cdot \ldots\cdot \mathcal O_Y(H_{p-1})\otimes \mathcal O_D\cdot \Delta_p)\\
\mbox{ where $s_p\in \ZZ_{>0}$ and $\Delta_p\in |\mathcal O_Y(s_pH_{p})\otimes \mathcal O_D|$ is such that $\Delta_p\cap\mbox{Ass}(\mathcal O_D)=\emptyset$,}\\
=\frac{1}{s_p}(\mathcal L_1\otimes \mathcal O_{\Delta_p}\cdot \ldots\cdot\mathcal L_{d-p}\otimes \mathcal O_{\Delta_p}\cdot \mathcal O_Y(H_2)\otimes \mathcal O_{\Delta_p}\cdot \ldots\cdot \mathcal O_Y(H_{p-1})\otimes \mathcal O_{\Delta_p}).
\end{array}
\end{equation}
Iterating, we obtain a $p$-cycle $W=\sum a_iV_i$ on $Y$, with $V_i$ closed $p$-dimensional subvarieties and $a_i$ positive rational numbers such that 
$$
(\mathcal L_1\cdot\ldots\cdot \mathcal L_{d-p}\cdot \mathcal O_Y(H_2)\cdot\ldots\cdot \mathcal O_Y(H_p)\cdot \mathcal O_Y(D))
=\sum a_i\sigma_{V_i}(\mathcal L_1,\ldots,\mathcal L_p)
$$
for all $\mathcal L_1,\ldots,\mathcal L_p\in M^1(Y)$.
We thus have that $\pi_Y(\alpha_1\cdot\ldots\cdot \alpha_p)\in \mbox{Psef}(L^p(Y))$. 

If $f:Z\rightarrow Y\in I(X)$, then $\alpha_1$ is represented in $M^1(Z)$
by the psef class $f^{*}(\mathcal N_1)$ and $\alpha_2,\ldots,\alpha_p$ are represented by the nef classes $f^{*}(\mathcal N_2),
\ldots, f^{*}(\mathcal N_p)$. Thus the above argument shows that $\pi_Z(\alpha_1\cdot\ldots\cdot \alpha_p)\in \mbox{Psef}(L^p(Z))$.
Since $I(X)$ is directed, and by Lemma \ref{Lemma1}, we have that $\pi_Z(\alpha_1\cdot\ldots\cdot \alpha_p)\in \mbox{Psef}(L^p(Z))$
for all $Z\in I(X)$. Thus $\alpha_1\cdot\ldots\cdot \alpha_p\in \mbox{Psef}(L^p(\mathcal X))$.
\end{proof}

\begin{Proposition}\label{PropNef} Suppose that $\alpha_i$ and $\alpha_i'$ for $1\le i\le p$ are nef classes in $M^1(\mathcal X)$,
and  that $\alpha_i\ge \alpha_i'$ for $i=1,\ldots,p$. Then
$$
\alpha_1\cdot\ldots\cdot\alpha_p\ge \alpha_1'\cdot\ldots\cdot\alpha_p'
$$
in $L^{d-p}(\mathcal X)$.
\end{Proposition}

\begin{proof} From  symmetry of the intersection product, and the assumption that $\alpha_{i}-\alpha_{i}'\ge 0$ for all $i$, we obtain from Lemma \ref{Lemma4} that
$$
\alpha_1\cdot\ldots\cdot\alpha_{i-1}\cdot(\alpha_{i}-\alpha_{i}')\cdot\alpha_{i+1}'\cdot\ldots\cdot\alpha_p'\ge 0
$$
for $1\le i\le p$. The proposition now follows from the multilinearity of the intersection product.
\end{proof}

\begin{Corollary}\label{Cor11} Suppose that $\alpha_1,\ldots,\alpha_d\in M^1(\mathcal X)$ are such that for some $p$ with $0\le p\le d$
$\alpha_i$ is nef for $i\le p$, and $\omega$ is a nef class in $M^1(\mathcal X)$ such that $\omega\pm \alpha_i$ is nef for each $i>p$.
Then
$$
|(\alpha_1\cdot\ldots\cdot \alpha_d)|\le C(\alpha_1\cdot\ldots\cdot \alpha_p\cdot\omega^{d-p})
$$
for some constant $C$ depending only on $(\omega^d)$.
\end{Corollary}

\begin{proof} Let $\beta_i=\alpha_i+\omega$ for $p< i\le d$. Expand 
$$
(\alpha_1\cdot\ldots\cdot \alpha_p\cdot \alpha_{p+1}\cdot\ldots\cdot \alpha_d)=(\alpha_1\cdot\ldots\cdot\alpha_p\cdot(\beta_{p+1}-\omega)\cdot\ldots\cdot (\beta_d-\omega))
$$
using multilinearity, to get an expression with terms
$$
(\alpha_1\cdot\ldots\cdot\alpha_p\cdot\beta_{j_1}\cdot\ldots\cdot\beta_{j_r}\cdot\omega^{d-p-r}).
$$
By our assumption, $0\le \beta_i\le 2\omega$ are nef for $p<i\le d$. The bounds thus follow from Proposition \ref{PropNef}.
\end{proof}

\section{The positive intersection product}\label{SecPos}

We  continue to assume that $X$ is a complete $d$-dimensional variety over a field $k$.

A partially ordered set is directed if any  two elements can be dominated by a third. A partially ordered set is filtered if any two elements dominate a third.

\begin{Lemma}\label{Lemma9}
Let $V$ be a Hausdorff topological vector space and $K$ a strict closed convex cone in $V$ with associated partial order relation $\le$.
Then any nonempty subset $S$ of $V$ which is directed with respect to $\le$ and  is contained in a compact subset of $V$
has a least upper bound with respect to $\le$ in $V$.
\end{Lemma}

\begin{proof} The set $S$ is a net in $V$ under the partial order $\le$. There exists an accumulation point of the net $S$ in $V$ since $S$ is contained in a compact subset of $V$ (c.f. Exercise 10, page 188  \cite{M}). Let $\gamma$ be an accumulation point. We will show that $\gamma$ is an upper bound of $S$. Suppose not. Then there exists $\delta\in S$ such that $\delta\not\le \gamma$. Thus $\gamma\not\in \delta+K$. $\delta+K$ is closed in $V$ since it is a translate of a closed set. Thus there exists an open neighborhood $U$ of $\gamma$ in $V$ such that $U\cap(\delta+K)=\emptyset$.
Since $\gamma$ is an accumulation point, there exists $\epsilon\in S$ such that $\delta\le\epsilon$ and $\epsilon\in U$. But
$\epsilon\in U\cap (\delta+K)=\emptyset$, a contradiction.

Suppose that $y\in V$ is an upper bound of $S$. Then we have  that $S\subset y-K$. Suppose that $\gamma\not\le y$. Then $\gamma\not\in y-K$ so there exists an open neighborhood $U$ of $\gamma$ in $V$ such that $U\cap (y-K)=\emptyset$. There exists $\epsilon\in S$ such that $\epsilon \in U$. $\epsilon\le y$ imples $\epsilon\in y-K$ so $U\cap (y-K)\ne\emptyset$, a contradiction. Thus
$\gamma\le y$ and  $\gamma$ is the (necessarily unique) least upper bound of $S$. In particular, the net $S$ converges to $\gamma$.
\end{proof}

\begin{Lemma}\label{Lemma10} Let $\alpha\in M^1(\mathcal X)$  be big. Then the set $\mathcal D(\alpha)$ of effective $\QQ$-divisors in $M^1(\mathcal X)$ such that $\alpha-D$ is nef is nonempty and filtered.
\end{Lemma}

\begin{proof} Let $Y\in I(X)$ be such that  $\alpha$ is represented in $M^1(Y)$ by a big divisor. Then there exists an effective $\QQ$-Cartier divisor $D$ in $M^1(Y)$ such that $\alpha-D$ is ample by (\ref{eq80}). Thus $D(\alpha)$ is nonempty.  

Let $D_1$, $D_2$ be two $\QQ$-Cartier divisors in $M^1(\mathcal X)$ such that $\alpha-D_1$ and $\alpha-D_2$ are nef. 
Let $Y\in I(X)$ be such that both $D_1$ and $D_2$ are represented there. There exists a positive integer $m$ such that $mD_1$ and $mD_2$ are integral Cartier  divisors on $Y$.
Let $\mathcal I=\mathcal O_Y(-mD_1)+\mathcal O_Y(-mD_2)$, an ideal sheaf on $Y$. Let $Z$ be the blow up of the ideal sheaf $\mathcal I$,
with natural morphism $f:Z\rightarrow Y$. $\mathcal I\mathcal O_Z$ is a locally principle ideal sheaf, so it determines an (integral) effective divisor $D$ with $\mathcal I\mathcal O_Z=\mathcal O_Z(-D)$. We have that $\mathcal O_Z(-f^*(mD_i))\subset\mathcal O_Z(-D)$ for $i=1,2$ so that $D':=\frac{1}{m}D\le f^*(D_i)$ for $i=1,2$. We must show that $\alpha-D'$ is nef.
Let $H_1,\ldots,H_r$ be ample divisors on $Y$ whose classes span $M^1(Y)$ as a real vector space. Given $\epsilon>0$, there exist
real numbers $a_i$ with $0\le a_i<\epsilon$ for all $i$, such that $\alpha + a_1H_1+\cdots+a_rH_r$ is a $\QQ$-divisor and  
$(\alpha + a_1H_1+\cdots+a_rH_r)-D_i$ are ample on $Y$ for $i=1,2$. There exists a positive integer $n$ which is divisible by $m$, and such that $n(\alpha + a_1H_1+\cdots+a_rH_r-D_i)$ are very ample integral divisors (so they are generated by global sections).
Let $\mathcal L=n(\alpha + a_1H_1+\cdots+a_rH_r)$, an integral divisor on $Y$. We have a surjection 
$$
(\mathcal O_X(-nD_1)\otimes\mathcal L)\bigoplus(\mathcal O_X(-nD_1)\otimes\mathcal L)\rightarrow \mathcal I^{\frac{n}{m}}\otimes \mathcal L.
$$
Thus $\mathcal I^{\frac{n}{m}}\otimes\mathcal L$ is generated by global sections, so
$$
(\mathcal I^{\frac{n}{m}}\mathcal O_Z)\otimes\mathcal L\cong \mathcal 
O_Z(n(\alpha+a_1f^*(H_1)+\cdots+\cdots+a_rf^*(H_r)-D'))
$$
is generated by global sections. Thus
$$
\alpha+a_1f^*(H_1)+\cdots+a_rf^*(H_r)-D'
$$
is nef. Since nefness is a closed condition, we have that $\alpha-D'$ is nef.
\end{proof}

\begin{Proposition}\label{Prop30} Suppose that $\alpha_1,\ldots,\alpha_p\in M^1(\mathcal X)$ are big. Let
$$
S=\left\{
\begin{array}{l}(\alpha_1-D_1)\cdot\ldots\cdot(\alpha_p-D_p)\in L^{d-p}(\mathcal X)\mbox{ such that } D_1,\ldots,D_p\in M^1(\mathcal X)\\ \mbox{ are effective $\QQ$-Cartier divisors  and $\alpha_i-D_i$ are nef for $1\le i\le p$}.
\end{array}
\right\}
$$
Then
\begin{enumerate}
\item[1)] $S$ is nonempty
\item[2)] $S$ is a directed set with respect to the partial order $\le$ on $L^{d-p}(\mathcal X)$
\item[3)] $S$ has a (unique) least upper bound with respect to $\le$ in $L^{d-p}(\mathcal X)$.
\end{enumerate}
\end{Proposition}

\begin{proof}  
$S$ is nonempty and directed by Lemma \ref{Lemma10} and Proposition \ref{PropNef}.  Let $Y\in I(X)$ be such that $\alpha_1,\ldots,\alpha_p$ are represented by elements of $M^1(Y)$ and let $\omega\in M^1(Y)$ be an ample class such that $\alpha_i\le \omega$ for all $i$. Then by Proposition \ref{PropNef},
$S$ is a subset of
$$
\{x\in L^{d-p}(\mathcal X)\mid 0\le x\le \omega^p\}
$$
which is compact by Lemma \ref{Lemma2}. The proposition now follows from Lemma \ref{Lemma9}.
\end{proof}

The following definition is well defined by virtue of Proposition \ref{Prop30}.

\begin{Definition}\label{DefPos} Let $\alpha_1,\ldots,\alpha_p\in M^1(\mathcal X)$ be big. Their positive intersection product
$$
<\alpha_1\cdot\ldots\cdot\alpha_p>\in L^{d-p}(\mathcal X)
$$
is defined as the least upper bound of the set of classes
$$
(\alpha_1-D_1)\cdot\ldots\cdot(\alpha_p-D_p)\in L^{d-p}(\mathcal X)
$$
where $D_i\in M^1(\mathcal X)$ are effective $\QQ$-Cartier classes such that $\alpha_i-D_i$ is nef.
\end{Definition}

\begin{Lemma}\label{Lemma71}
Suppose that $\alpha_1,\ldots,\alpha_p,\beta_1,\ldots,\beta_p\in M^1(\mathcal X)$ are big, with $0\le p\le d$.
Then
$$
<\alpha_1\cdot\ldots\cdot\alpha_p>\le <(\alpha_1+\beta_1)+\cdot\ldots\cdot (\alpha_p+\beta_p)>.
$$
\end{Lemma}

\begin{proof} Let $S$ be the set of Proposition \ref{Prop30} defining $<\alpha_1\cdot\ldots\cdot \alpha_p>$ and let $T$ be the set defining 
$$
<(\alpha_1+\beta_1)\cdot\ldots\cdot (\alpha_p+\beta_p)>.
$$
Suppose that 
$$
\sigma=(\alpha_1-D_1)\cdot\ldots\cdot(\alpha_p-D_p)\in S.
$$
Let $Y\in I(X)$ be such that $\alpha_1,\ldots,\alpha_p,D_1,\ldots,D_p,\beta_1,\ldots,\beta_p$ are represented in $M^1(Y)$. For $1\le i\le p$ we have $\beta_i=H_i+E_i$ where $H_i$ is an ample $\RR$-divisor and $E_i$ is an effective $\QQ$-divisor (by (\ref{eq80})). Thus
$$
\tau=((\alpha_1+\beta_1)-(D_1+E_1))\cdot\ldots\cdot((\alpha_p+\beta_p)-(D_p+E_p))\in T
$$
and $\sigma\le \tau$ by Proposition \ref{PropNef}. Thus $<(\alpha_1+\beta_1)\cdot\ldots\cdot(\alpha_p+\beta_p)>$ is an upper bound for $S$.
\end{proof}

\begin{Lemma}\label{Lemma11} Suppose that $\alpha_1,\ldots,\alpha_p\in M^1(\mathcal X)$ are big. Suppose that $Y\in I(X)$, $|*|$ is a norm on $M^1(Y)$ giving the Euclidean topology, and $\epsilon$ is a positive real number. Then there exist $\beta_1,\ldots,\beta_p\in M^1(\mathcal X)$ which are nef and satisfy $\beta_i\le \alpha_i$ for all $i$ such that
$$
|(<\alpha_1\cdot\ldots\cdot \alpha_p>-\beta_1\cdot\ldots\cdot \beta_p)(\mathcal L_1,\ldots, \mathcal L_{d-p})|<\epsilon |\mathcal L_1|\cdots|\mathcal L_{d-p}|
$$
for all $\mathcal L_1,\ldots,\mathcal L_{d-p}\in M^1(Y)$.
\end{Lemma}

\begin{proof} $<\alpha_1\cdot\ldots\cdot \alpha_p>$ is the limit point in $L^{d-p}(\mathcal X)$ of the net 
$$
S=\{\beta_1\cdot\ldots\cdot\beta_p\in L^{d-p}(\mathcal X)\mid \mbox{ each $\beta_i$ is nef and $D_i=\alpha_i-\beta_i$ is $\QQ$-Cartier}\}.
$$
There exists an open neighborhood $U$ of $\pi_Y(<\alpha_1\cdot\ldots\cdot\alpha_p>)$ in $L^{d-p}(Y)$ such that 
$$
||A-\pi_Y(<\alpha_1\cdot\ldots\cdot\alpha_p>)||<\epsilon\mbox{ for }A\in U
$$
where $||*||$ is the norm on $L^{d-p}(Y)$ defined before Lemma \ref{Lemma60}.
Thus there exists an element $\beta_1\cdot\ldots\cdot\beta_p\in S\cap \pi_Y^{-1}(U)$ since $<\alpha_1\cdot\ldots\cdot\alpha_p>$ is the limit  point of $S$, so $\pi_Y(\beta_1\cdot\ldots\cdot\beta_p)$ has the desired property.
\end{proof}

\begin{Proposition}\label{PropCont} The  map $\mbox{Big}^p(\mathcal X)\rightarrow L^{d-p}(\mathcal X)$ defined by
$$
(\alpha_1,\ldots,\alpha_p)\mapsto <\alpha_1\cdot\ldots\cdot \alpha_p>
$$
 is continuous.
\end{Proposition}

\begin{proof} Since the topologies on $\mbox{Big}^p(\mathcal X)$ and $L^{d-p}(\mathcal X)$ are respectively the strong and weak topologies,
it suffices to show that for each $Y\in I(X)$, the map 
$$
\mbox{Big}(Y)^p\stackrel{\rho_Y}{\rightarrow} \mbox{Big}^p(\mathcal X)\rightarrow L^{d-p}(\mathcal X) \stackrel{\pi_Y}{\rightarrow} L^{d-p}(Y)
$$
is continuous. Let  $|\,\,\,|$ be a norm on $M^1(Y)$ and $||\,\,\,||$ be the norm on $L^{d-p}(Y)$ defined before Lemma \ref{Lemma60}  giving the Euclidean topologies.

Let $\mathcal L_i\in \mbox{Big}(Y)$ for $1\le i\le p$, and suppose that $\epsilon$ is a positive real number. Let 
$$
\Omega=\{z\in L^{d-p}(Y)|0\le z\le  <\mathcal L_1\cdot\ldots\cdot \mathcal L_p>\}.
$$
$\Omega$ is compact by Lemma \ref{Lemma7}.  Let
$$
u=\max\{||z||\mid z\in \Omega\}.
$$
Choose a rational number $\lambda$ with $0<\lambda<1$ so that
$$
((1+\lambda)^p-(1-\lambda)^p)u<\frac{\epsilon}{2}
$$
and
$$
(1-(1-\lambda)^p)||<\mathcal L_1\cdot\ldots\cdot\mathcal L_p>||<\frac{\epsilon}{2}.
$$
Since $\lambda \mathcal L_i\in \mbox{Big}(Y)$, there exists $\delta>0$ such that if $\gamma_i\in M^1(Y)$ and $|\gamma_i|<\delta$, then 
$\lambda\mathcal L_i\pm \gamma_i\in \mbox{Big}(Y)$ for all $i$. Hence  
$$
(1-\lambda)\mathcal L_i\le \mathcal L_i+\gamma_i\le (1+\lambda)\mathcal L_i,
$$
and thus by Lemma \ref{Lemma71} and Definition \ref{DefPos}, 
$$
(1-\lambda)^p<\mathcal L_1\cdot\ldots\cdot \mathcal L_p>\le <(\mathcal L_1+\gamma_1)\cdot\ldots\cdot (\mathcal L_p+\gamma_p)>
\le (1+\lambda)^p<\mathcal L_1\cdot\ldots\cdot \mathcal L_p>
$$
and
$$
\begin{array}{lll}
((1-\lambda)^p-1)<\mathcal L_1\cdot\ldots\cdot \mathcal L_p>
&\le&<(\mathcal L_1+\gamma_1)\cdot\ldots\cdot (\mathcal L_p+\gamma_p)>-<\mathcal L_1\cdot\ldots\cdot\mathcal L_p>\\
&\le& ((1+\lambda)^p-1)<\mathcal L_1\cdot\ldots\cdot \mathcal L_p>.
\end{array}
$$
Let 
$$
v=<(\mathcal L_1+\gamma_1)\cdot\ldots\cdot (\mathcal L_p+\gamma_p)>-<\mathcal L_1\cdot\ldots\cdot\mathcal L_p>.
$$
$$
v\in ((1-\lambda)^p-1)<\mathcal L_1\cdot\ldots\cdot\mathcal L_p>+(([(1+\lambda)^p-1]+[1-(1-\lambda)^p])\Omega
$$
implies $||v||<\epsilon$ by the triangle inequality.

\end{proof}

\begin{Definition}\label{Def2} Suppose that $\alpha_1,\ldots,\alpha_p\in \mbox{Psef}(\mathcal X)$. Then their positive intersection product
$$
<\alpha_1\cdot\ldots\cdot\alpha_p>\in L^{d-p}(\mathcal X)
$$
is defined as the limit
$$
\lim_{\epsilon\rightarrow 0+}<(\alpha_1+\epsilon\omega)\cdot\ldots\cdot(\alpha_p+\epsilon \omega)>
$$
where $\omega\in M^1(\mathcal X)$ is any big class.
\end{Definition}

\begin{Lemma} Definition \ref{Def2} is well defined.
\end{Lemma}

\begin{proof} Suppose that $\omega\in M^1(\mathcal X)$ is big. The set
$$
S_{\omega}=\{-<(\alpha_1+t\omega)\cdot\ldots\cdot(\alpha_p+t\omega)>\mid 0<t\le 1\}
$$
is a directed set under $\le$ by Lemma \ref{Lemma71}, and it is contained in the compact set
$$
-\{x\in L^{d-p}(\mathcal X)\mid 0\le x\le <(\alpha_1+\omega)\cdot\ldots\cdot(\alpha_p+\omega)>,
$$
so it has a least upper bound $y$ in $L^{d-p}(\mathcal X)$ by Lemma \ref{Lemma9}. Setting $z=-y$, we have
$$
z=\lim_{\epsilon\rightarrow 0^+}<(\alpha_1+\epsilon\omega)\cdot\ldots\cdot(\alpha_p+\epsilon\omega)>
$$
is well defined.

We have equality of sets 
$$
E=\{x\in L^{d-p}(\mathcal X)\mid 0\le x\le z\}=\cap_{0<\epsilon\le 1}C_{\epsilon}
$$
where 
$$
C_{\epsilon}=\{x\in L^{d-p}(\mathcal X)\mid 0\le x\le <(\alpha_1+\epsilon\omega)\cdot\ldots\cdot(\alpha_p+\epsilon\omega)>\}.
$$
In fact, $E\subset \cap_{0<\epsilon\le 1}C_{\epsilon}$ since $-z$ is a least upper bound for $S_{\omega}$, and 
$\lambda\in\cap_{0<\epsilon\le 1}C_{\epsilon}$ implies $-\lambda$ is an upper bound for $S_{\omega}$, so
$-z\le-\lambda$ and thus $\lambda\le z$.

Suppose that $\omega'\in M^1(\mathcal X)$ is another big class. Let
$$
z'=\lim_{\epsilon\rightarrow 0^+}<(\alpha_1+\epsilon\omega')\cdot\ldots\cdot(\alpha_p+\epsilon\omega')>.
$$
Suppose that $z\ne z'$. We will derive a contradiction. Then we either have that $z\not\le z'$ or $z'\not\le z$. Without loss of generality, we may assume that $z'\not\le z$. Thus there exists a positive real number $\epsilon$ such that $z'\not\in C_{\epsilon}$.
$C_{\epsilon}$ is closed in $L^{d-p}(\mathcal X)$ since $C_{\epsilon}$ is compact and $L^{d-p}(\mathcal X)$ is Hausdorff.
Let $U$ be an open neighborhood of $z'$ in $L^{d-p}(\mathcal X)$. If $\delta>0$ is sufficiently small, then $\epsilon\omega-\delta\omega'$ is big (by (\ref{eq80})) and there exists such a $\delta$ with 
$$
<(\alpha_1+\delta\omega')\cdot\ldots\cdot(\alpha_p+\delta\omega')>\in U.
$$
$$
<(\alpha_1+\delta\omega')\cdot\ldots\cdot (\alpha_p+\delta\omega')>\le <(\alpha_1+\epsilon\omega)\cdot\ldots\cdot(\alpha_p+\epsilon\omega)>
$$
by Lemma \ref{Lemma71}. Thus $U\cap C_{\epsilon}\ne \emptyset$.  Since this is true for all open neighborhoods $U$ of $z'$, $z'$ is in the closed set $C_{\epsilon}$, giving a contradiction. Thus $z=z'$.
\end{proof}

\begin{Remark} In Example 3.8 \cite{BFJ}, it is shown that the positive intersection product is not continuous up to the boundary of the psef cone in general, even on a nonsingular surface.
\end{Remark}

\begin{Proposition}\label{Prop11} If $\alpha_1,\ldots,\alpha_p\in M^1(\mathcal X)$ are nef, then
$$
<\alpha_1\cdot\ldots\cdot\alpha_p>=\alpha_1\cdot\ldots\cdot\alpha_p.
$$\end{Proposition} 

\begin{proof} When the $\alpha_i$ are nef and big we can take $D_i=0$ in  Definition \ref{DefPos}, from which the statement follows.
The general case follows from the big case, continuity of the (usual) intersection product (Lemma \ref{Intcont}),  and Definition \ref{Def2} by taking $\omega$ to be nef and big.
\end{proof}

\begin{Proposition}\label{Prop10} Suppose that $\alpha_1,\ldots,\alpha_p\in \mbox{Big}(\mathcal X)$. Then $<\alpha_1\cdot\ldots\cdot\alpha_p>$ is the least upper bound of all intersection products
$\beta_1\cdot\ldots\cdot \beta_p$ in $L^{d-p}(\mathcal X)$ with $\beta_i$ a nef class such that $\beta_i\le\alpha_i$.
In particular, $<\alpha^p>$ is the least upper bound of $(\beta^p)$ such that $\beta$ is nef and $\beta\le\alpha$.

\end{Proposition}

\begin{proof} Let $S$ be the directed set of Proposition \ref{Prop30} and let $T$ be the set
$$
T=\{\beta_1\cdot\ldots\cdot\beta_p\mid \beta_i\in M^1(\mathcal X)\mbox{ are nef and }\beta_i\le \alpha_i\mbox{ for $1\le i\le p$}\}.
$$
Suppose that $\beta_1,\ldots,\beta_p\in M^1(\mathcal X)$ are nef with $\beta_i\le\alpha_i$. Let $Y \in I(X)$ be such that
$\alpha_1,\ldots,\alpha_p,\beta_1,\ldots,\beta_p$ are represented on $Y$. Let $H_1,\ldots,H_s$ be ample, integral divisors on $Y$ which generate $M^1(Y)$ as an $\RR$-vector space. Given a real number $\delta>0$, there exists $0<\epsilon_j^i<\delta$ such that
$$
D_i:=\alpha_i-\beta_i+\epsilon_1^iH_1+\cdots+\epsilon_s^iH_s
$$
is represented in $M^1(Y)$ by an effective and big $\QQ$-divisor. $(\alpha_i+\epsilon_1^iH_1+\cdots+\epsilon_s^iH_s)-D_i=\beta_i$ is nef, so that
$$
\beta_1\cdot\ldots\cdot\beta_p\le
<(\alpha_1+\epsilon_1^1H_1+\cdots+\epsilon_s^1H_s)\cdot\ldots\cdot (\alpha_p+\epsilon_1^pH_1+\cdots+\epsilon_s^pH_s)>
$$
by Definition \ref{DefPos}.
We have a continuous  map
$$
\Lambda:\RR^{ps}\rightarrow (M^1(Y))^p
$$
defined by 
$$
\Lambda(\epsilon_1^1,\ldots,\epsilon_s^1,\ldots, \epsilon_1^p,\ldots,\epsilon_s^p)
=(\alpha_1+\epsilon_1^1H_1+\cdots+\epsilon_s^1H_s,\ldots,\alpha_p+\epsilon_1^pH_1+\cdots+\epsilon_s^pH_s).
$$
$U=\Lambda^{-1}(\mbox{Big}(Y)^p)$ is an open subset of $\RR^{ps}$ containing the origin. 
composing with the continuous map $\mbox{Big}(Y)^p\rightarrow L^{d-p}(Y)$ defined by 
$(z_1,\ldots,z_p)\mapsto \pi_Y(<z_1\cdot\ldots\cdot z_p>)$, we obtain a continuous map $U\rightarrow L^{d-p}(Y)$. Thus
$$
\lim_{\epsilon_i^j\rightarrow 0}\pi_Y(<((\alpha_1+\epsilon_1^1H_1+\cdots+\epsilon_s^1H_s)\cdot\ldots\cdot
(\alpha_p+\epsilon_1^pH_1+\cdots+\epsilon_s^pH_s)>)=\pi_Y(<\alpha_1\cdot\ldots\cdot \alpha_p>).
$$
Since 
$$
\pi_Y(<((\alpha_1+\epsilon_1^1H_1+\cdots+\epsilon_s^1H_s)\cdot\ldots\cdot
(\alpha_p+\epsilon_1^pH_1+\cdots+\epsilon_s^pH_s)>)-\pi_Y(\beta_1\cdot\ldots\cdot\beta_p)
$$
is in the closed subset $\mbox{Psef}(L^{d-p}(Y))$ for all $\epsilon_i^j>0$, we have that
$$
\pi_Y(\beta_1\cdot\ldots\cdot\beta_p)\le\pi_Y(<\alpha_1\cdot\ldots\cdot\alpha_p>).
$$
Thus $\beta_1\cdot\ldots\cdot\beta_p\le <\alpha_1\cdot\ldots\cdot \alpha_p>$.
We have $S\subset T$ and the least upper bound $<\alpha_1\cdot\ldots\cdot\alpha_p>$ of $S$ is an upper bound of $T$. Thus $<\alpha_1\cdot\ldots\cdot \alpha_p>$ is the least upper bound of $T$. 
\end{proof}

\begin{Lemma}\label{LemmaSym} For $\alpha_1,\ldots,\alpha_p\in \mbox{Psef}(\mathcal X)$, the positive intersection product 
$$
(\alpha_1,\ldots,\alpha_p)\mapsto <\alpha_1\cdot\ldots\cdot\alpha_p>\in L^{d-p}(\mathcal X)
$$
is symmetric, homogeneous of degree 1 and super-additive in each variable.
\end{Lemma}

\begin{proof} First suppose that $\alpha_1,\ldots,\alpha_p\in \mbox{Big}(\mathcal X)$. The only part that does not follow directly from the definition of the positive intersection product is the statement on homogeneity for irrational scalars. By symmetry, it suffices to prove homogeneity in the first variable. Suppose that
$\alpha_1,\ldots,\alpha_p\in \mbox{Big}(\mathcal X)$. The map $\phi:\RR_{>0}\rightarrow L^{d-p}(\mathcal X)$ given by
$$
\lambda\mapsto <\lambda\alpha_1\cdot\alpha_2\cdot\ldots\cdot \alpha_p>
$$ 
is continuous by Proposition \ref{PropCont}; in
fact it has a natural factorization by continuous maps
$$
\RR_{>0}\rightarrow \mbox{Big}(Y)^p\stackrel{\rho_Y}{\rightarrow}\mbox{Big}^p(\mathcal X)\rightarrow L^{d-p}(\mathcal X)
$$
if $\alpha_1,\ldots,\alpha_p$ are represented in $\mbox{Big}(Y)$.
Since $L^{d-p}(\mathcal X)$ is a topological vector space, the map $\psi:\RR_{>0}\rightarrow L^{d-p}(\mathcal X)$ defined by
$\lambda\mapsto \lambda<\alpha_1\cdot\ldots\cdot\alpha_p>$ is continuous. Since $\phi$ and $\psi$ agree on the positive rational numbers, and $L^{d-p}(\mathcal X)$ is Hausdorff, we have that $\phi=\psi$.

Symmetry on $\mbox{Psef}(\mathcal X)$ follows from the big case and Definition \ref{Def2}.
It suffices to establish super additivity in the first variable. Suppose that $\alpha_1,\alpha_1',\alpha_2,\ldots\alpha_p\in \mbox{Psef}(\mathcal X)$. There exists $Y\in I(X)$ such that $\alpha_1,\alpha_1',\alpha_2,\ldots\alpha_p$ are represented in $\mbox{Psef}(Y)$. Let $\omega$ be an ample divisor on $Y$. Define
$$
\phi:\RR_{>0}\rightarrow L^{d-p}(\mathcal X)
$$
by
$$
\begin{array}{lll}
\phi(t)&=&<(\alpha_1+t\omega)+(\alpha_1'+t\omega)\cdot (\alpha_2+t\omega)\cdot\ldots\cdot (\alpha_p+t\omega)>\\
&&-<(\alpha_1+t\omega)\cdot(\alpha_2+t\omega)\cdot\ldots\cdot(\alpha_p+t\omega)>\\
&&-<(\alpha_1'+t\omega)\cdot(\alpha_2+t\omega)\cdot\ldots\cdot(\alpha_p+t\omega)>.
\end{array}
$$
$\phi$ is continuous by Proposition \ref{PropCont}, and $\phi(\RR_{>0})$ is contained in the closed set $K$ of pseudoeffective classes.
Taking the limit as $t\rightarrow 0^+$, we have by Definition \ref{Def2} that 
$$
<(\alpha_1+\alpha_1')\cdot\ldots\cdot\alpha_p>-<\alpha_1\cdot\alpha_2\cdot\ldots\cdot\alpha_p>-<\alpha_1.\cdot\alpha_2\cdot\ldots\cdot\alpha_p>\in K.
$$
\end{proof}

\section{Volume}\label{SecVolume}
In this section we continue to assume that $X$ is a complete $d$-dimensional variety over a field $k$.

\begin{Theorem}\label{FApp}(Fujita Approximation)
Suppose that $D$ is a big Cartier divisor on a complete variety $X$  of dimension $d$ over a field $k$, and $\epsilon>0$ is given. Then there exists a projective variety $Y$ with a birational morphism $f:Y\rightarrow X$, a nef and big $\QQ$-divisor $N$ on $Y$, and an effective $\QQ$-divisor $E$ on $Y$ such that there exists $n\in \ZZ_{>0}$ so that $nD$, $nN$ and $nE$ are Cartier divisors with
 $f^*(nD)\sim nN+nE$, where $\sim$ denotes linear equivalence, and
 $$
 {\rm vol}_Y(N)\ge {\rm vol}_X(D)-\epsilon.
 $$
 \end{Theorem}
 
 \begin{proof} By taking a Chow cover by a birational morphism, which is an isomorphism in codimension one, we may assume that $X$ is projective over $k$. This theorem was proven over an algebraically closed field of characteristic zero by Fujita \cite{F2} (c.f. Theorem 10.35 \cite{L}). It is proven in Theorem 3.4 and Remark 3.4 \cite{LM} over an arbitrary algebraically closed field (using Okounkov bodies) and by Takagi \cite{Ta} using de Jong's alterations \cite{DJ}. 
 
We give a proof for an arbitrary field.  The conclusions of Theorem 3.3 \cite{LM} over an arbitrary field follow from Theorem 7.2 and formula (45) of \cite{C}, taking  the $L_n$ of Theorem 7.2 \cite{C} to be the $H^0(X,\mathcal O_X(nD))$ of Theorem 3.3 \cite{LM}. $m=1$ in Theorem 7.2 \cite{C} since $D$ is big.  Then the $V_{k,p}$ of Theorem 3.3 \cite{LM} are the $L_{kp}^{[p]}$ of the proof of Theorem 7.2
\cite{C}.

The proof of Remark 3.4 \cite{LM} is valid over an arbitrary field, using the strengthened form of Theorem 3.3 \cite{LM} given above, from which the
approximation theorem follows.
 
 The following theorem is proven in Theorem 3.1 \cite{BFJ} when $k$ is algebraically closed of characteristic zero.
 
 \end{proof}

\begin{Theorem}\label{TheoremVol} Suppose that $X$ is a complete variety over a field $k$ and $\mathcal L$ is a big line bundle on $X$. Then 
$$
{\rm vol}_X(\mathcal L)=<\mathcal L^d>.
$$
\end{Theorem}

\begin{proof} Suppose that $Y\in I(X)$ and $D$ is an effective $\QQ$-divisor on $Y$ such that $\mathcal L-D$ is nef and $\mathcal L\ge \mathcal L-D$. There exists $t\in \ZZ_{+}$ such that $tD$ is an effective Cartier divisor. For $m\in \ZZ_+$, we have an inclusion
of $\mathcal O_Y$-modules
$$
f^*(\mathcal L^{mt})\otimes \mathcal O_Y(-mtD)\rightarrow f^*(\mathcal L^{mt}).
$$
Thus
$$
\dim_k\Gamma(Y,f^*(\mathcal L^{mt})\otimes \mathcal O_Y(-mtD))\le \dim_k\Gamma(Y,f^*(\mathcal L^{mt})).
$$
We have a short exact sequence of sheaves of $\mathcal O_X$-modules
$$
0\rightarrow \mathcal O_X\rightarrow f_*\mathcal O_Y
\rightarrow \mathcal G\rightarrow 0
$$
where the support of $\mathcal G$ has dimension less than $d$ since $f$ is birational. Tensoring with $\mathcal L^{mt}$ and taking global sections gives us that
$$
\lim_{m\rightarrow\infty}\frac{\dim_k\Gamma(X,\mathcal L^{mt})}{m^d}=\lim_{m\rightarrow\infty}\frac{\dim_k\Gamma(Y,f^*(\mathcal L^{mt}))}{m^d}.
$$
Thus 
$$
\begin{array}{lll}
\mbox{vol}_X(\mathcal L)&=&\frac{1}{t^d}\mbox{vol}_X(t\mathcal L)\ge \lim_{m\rightarrow\infty}\frac{d!}{t^d}\frac{\dim_k\Gamma(Y,f^*(\mathcal L^{mt})\otimes\mathcal O_Y(-mtD))}{m^d}\\
&=& \frac{\mbox{vol}_Y(f^*(\mathcal L^t)\otimes\mathcal O_Y(-tD))}{t^d}\\
&=& \frac{(f^*(\mathcal L^t)-tD)^d}{t^d}\\
&=&(f^*(\mathcal L)-D)^d,
\end{array}
$$
where the first equality is by (\ref{eq85}) and the third line follows from Fujita's vanishing theorem 
(Theorem 6.2 \cite{F} when $k$ is algebraically closed and $X$ is a proper scheme over $k$. The statement for an arbitrary field follows from flat base change of $X\times_k{\overline k}$, where $\overline k$ is an algebraic closure of $k$ and Proposition III.9.3 \cite{H}.
See also Corollary 1.4.41 and Remark 1.4.36 \cite{L}). Thus $\mbox{vol}_X(\mathcal L)\ge <\mathcal L^d>$ by the definition of the positive intersection product.

Suppose that $\epsilon>0$. There exists a Fujita approximation $Y\rightarrow X$ where $Y$ is a projective $k$-variety, $f$ is a birational morphism and $f^*\mathcal L=A+E$
where $A$ is nef and big and $E$ is an effective $\QQ$-divisor, with $\mbox{vol}_Y(A)>\mbox{vol}_X(\mathcal L)-\epsilon$, by Theorem \ref{FApp}.
We have that
$$
\mbox{vol}_Y(A)=(A^d)\le <\mathcal L^d>
$$
 by Fujita's vanishing theorem. Thus $\mbox{vol}_X(\mathcal L)\le <\mathcal L^d>+\epsilon$.

\end{proof}

\begin{Theorem} The function $\alpha\mapsto <\alpha^d>$ is  continuous on $\mbox{Psef}(M^1(\mathcal X))$ and vanishes
on its boundary, and only there.
\end{Theorem}

\begin{proof} This follows from Theorem \ref{Volcont}, Lemma \ref{LemmaSym}, Theorem \ref{TheoremVol} and Definition \ref{Def2}.
\end{proof}

\begin{Proposition}\label{Prop20} Suppose that $A,B\in M^1(\mathcal X)$ are nef. Then
$$
{\rm vol}(A-B)\ge (A^d)-d(A^{d-1}\cdot B).
$$
\end{Proposition}

\begin{proof}  $A,B$ are represented by nef elements of $M^1(Y)$ for some $Y\in I(X)$. When $A,B$ are 
$\QQ$-Cartier this is Example 2.2.28 \cite{L} and (\ref{eq85}). Suppose that $A,B$ are $\RR$-divisors. Let $H_1,\ldots,H_r$ be ample classes in $M^1(Y)$
which generate $M^1(Y)$ as a real vector space. Given a positive real number $\epsilon$, there exist $0\le s_i<\epsilon$ and 
$0\le t_i<\epsilon$ for $1\le i\le r$ such that $A+s_1H_1+\cdots+s_rH_r$ and $B+t_1H_1+\cdots+t_rH_r$ are represented by nef $\QQ$-divisors, so that the formula holds for these divisors. Define $\phi:\RR^{2r}\rightarrow \RR$ by
$$
\begin{array}{l}
(s_1,\ldots,s_r,t_1,\ldots,t_r)\mapsto {\rm vol}((A+s_1H_1+\cdots+s_rH_r)-(B+t_1H_1+\cdots+t_rH_r))\\
-(A+s_1H_1+\cdots+s_rH_r)^d+d(A+s_1H_1+\cdots+s_rH_r)^{d-1}\cdot(B+t_1H_1+\cdots+t_rH_r).
\end{array}
$$
$\phi$ is a composition of continuous maps so it is continuous. $Z=\phi^{-1}(\{x\in \RR\mid x\ge 0\})$ is a closed subset of $\RR^{2r}$. 
Thus $0\in Z$ and the formula follows, since any neighborhood of 0 in $\RR^{2r}$ contains points of $Z$.
\end{proof}

\begin{Corollary}\label{Cor21} Suppose that $\beta,\gamma\in M^1(\mathcal X)$ and $\beta$ is nef. If $\omega\in M^1(\mathcal X)$ is a fixed nef and big class such that $\beta\le \omega$ and $\omega\pm\gamma$ are nef, then there exists a positive real number $C$, depending only on $(\omega^d)$, such that
$$
{\rm vol}(\beta+t\gamma)\ge (\beta^d)+dt(\beta^{d-1}\cdot\gamma)-Ct^2
$$
for every $0\le t\le 1$. 
\end{Corollary}

\begin{proof} 
The expansion
$$
(\beta^d+t\gamma)^d-[\beta^d+dt(\beta^{d-1}\cdot\gamma)]=\sum_{i=2}^d\binom{d}{i}t^i(\beta^{d-i}\cdot\gamma^i)
$$
has a lower bound on $0\le t\le 1$ in terms of $(\beta^{d-i}\cdot \omega^i)$ for $2\le i\le d$ by Corollary \ref{Cor11},
and thus a lower bound in terms of $(\omega^d)$ by Proposition \ref{PropNef}, since $\beta\le \omega$. Thus there exists a positive constant $C_1$, depending only on $(\omega^d)$, such that
\begin{equation}\label{eq40}
(\beta+t\gamma)^d\ge \beta^d+dt(\beta^{d-1}\cdot \gamma)-C_1t^2.
\end{equation}
Write
$$
\beta+t\gamma=A-B
$$
as the difference of two nef classes $A=\beta+t(\gamma+\omega)$ and $B=t\omega$.
$$
(A-B)^d-[(A^d)-d(A^{d-1}\cdot B)]
$$
has an upper bound  in terms of $(A^{d-1}\cdot B^i)$ for $2\le i\le d$, so it has an upper bound on $0\le t\le 1$ in terms of $(A^{d-1}\cdot \omega^i)t^i$ for $2\le i\le d$, and thus an upper bound in terms of $(\omega^d)$, since $A\le 3\omega$ and by Proposition
\ref{PropNef}, so we have
\begin{equation}\label{eq41}
A^d-d(A^{d-1}\cdot B)\ge (A-B)^d-C_2t^2,
\end{equation}
where $C_2$ is a positive constant which only depends on $(\omega^d)$. The corollary now follows from Proposition \ref{Prop20}, (\ref{eq41}) and (\ref{eq40}).

\end{proof}

Theorem \ref{TheoremA} is proven in Theorem A of \cite{BFJ} when $k$ is algebraically closed of characteristic zero. The fact that volume is continuously differentiable on the big cone is proven by Lazarsfeld and Mustata over an algebraically closed field of any characteristic in Remark 4.2.7 \cite{LM}.

In Example 2.7 \cite{ELMNP1}, it is shown that vol is not twice differentiable on the big cone of the blow up of $\PP^2$ at a rational point.

\begin{Theorem}\label{TheoremA} Suppose that $X$ is a complete $d$-dimensional variety over a field $k$. Then the volume function is $\mathcal C^1$ differentiable on the big cone of $N^1(X)$. If $\alpha\in N^1(X)$ is big and $\gamma\in N^1(X)$ is arbitrary, then
$$
\left. \frac{d}{dt}\right|_{t=0}{\rm vol}(\alpha+t\gamma)=d<\alpha^{d-1}>(\gamma).
$$
\end{Theorem}

\begin{proof} Fix a  nef and sufficiently big class $\omega\in M^1(\mathcal X)$ such that $\alpha\le \omega$ and $\omega\pm \gamma$ are nef.
Suppose that $\beta\le\alpha$ is nef. Then certainly $\beta\le \omega$, so it follows from Corollary \ref{Cor21} that
$$
{\rm vol}(\alpha+t\gamma)\ge{\rm vol}(\beta+t\gamma)\ge (\beta^d)+dt(\beta^{d-1}\cdot\gamma)-Ct^2
$$
for every $0\le t\le 1$ and some constant $C$ which only depends on $(\omega^d)$. By Lemma \ref{Lemma11} we thus have that
\begin{equation}\label{eq12}
{\rm vol}(\alpha+t\gamma)\ge {\rm vol}(\alpha)+dt<\alpha^{d-1}>(\gamma)-Ct^2
\end{equation}
for $0\le t\le 1$, and in fact for $-1\le t\le 1$, since the identity holds with $\gamma$ replaced by $-\gamma$.
There exists a possibly larger constant $C'$, depending only on $(\omega^d)$, such that (\ref{eq12}) holds for all big $\tilde\alpha$
such that $\tilde\alpha\le 2\omega$ and $-1\le t\le 1$; that is,
$$
{\rm vol}(\tilde \alpha-t\gamma)\ge{\rm vol}(\tilde\alpha)-dt<\tilde\alpha^{d-1}>(\gamma)-C't^2
$$
for $-1\le t\le 1$.   Since $\alpha=(\alpha+t\gamma)-t\gamma$ with $\alpha+t\gamma\le 2\omega$, we have that
\begin{equation}\label{eq13}
{\rm vol}(\alpha)\ge{\rm vol}(\alpha+t\gamma)-dt<(\alpha +t\gamma)^{d-1}>(\gamma)-C't^2
\end{equation}
for $-1\le t\le 1$. By (\ref{eq80}), there exists $0<\lambda<1$ such that $\alpha+t\gamma$ is big for $-\lambda<t<\lambda$.
The map $t\mapsto <(\alpha+t\gamma)^{d-1}>(\gamma)$ is a continuous map from the interval $(-\lambda,\lambda)$ in $\RR$ to $\RR$, as it can be factored by  continuous maps
$$
(-\lambda,\lambda)\rightarrow \mbox{Big}(X)^{d-1}\stackrel{\rho_X}{\rightarrow} \mbox{Big}^{d-1}(\mathcal X)\stackrel{<*\cdot,\ldots,\cdot*>}{\rightarrow}
L^1(\mathcal X)\stackrel{\pi_X}{\rightarrow} L^1(X)\stackrel{ \gamma}{\rightarrow}\RR.
$$
We thus have that
\begin{equation}\label{eq14}
\lim_{t\rightarrow 0}<(\alpha+t\gamma)^{d-1}>(\gamma)=<\alpha^{d-1}>(\gamma).
\end{equation}
From (\ref{eq12}), (\ref{eq13}) and (\ref{eq14}) we obtain the limit of the conclusions of the theorem.

For fixed $\gamma$, the map $<\alpha^{d-1}>(\gamma):\mbox{Big}(X)\rightarrow \RR$ is a composition of continuous  maps
 so ${\rm vol}$ is $\mathcal C^1$ on $\mbox{Big}(X)$.

\end{proof}

\section{Inequalities}\label{SecIneq}

In this section we suppose that $X$ is a complete $d$-dimensional variety over a field $k$.

\begin{Theorem}\label{KT1} Suppose that $\alpha_1,\ldots,\alpha_d\in N^1(X)$ are nef. Then for every $1\le p\le d$, we have
\begin{equation}\label{eq102}
(\alpha_1\cdot\ldots\cdot \alpha_d)\ge (\alpha_1^p\cdot\alpha_{p+1}\cdot\ldots\cdot\alpha_d)^{\frac{1}{p}}\cdots
(\alpha_p^p\cdot\alpha_{p+1}\cdot\ldots\cdot\alpha_d)^{\frac{1}{p}}.
\end{equation}
In particular,
\begin{equation}\label{eq103}
(\alpha_1\cdot\ldots\cdot\alpha_d)\ge (\alpha_1^d)^{\frac{1}{d}}\cdots(\alpha_d^d)^{\frac{1}{d}}.
\end{equation}
\end{Theorem}

\begin{proof}
This is proved over an algebraically closed field of characteristic zero in Variant 1.6.2 of \cite{L} (which is true in positive characteristic by Remark 1.6.5 \cite{L}). Lazarsfeld refers to Beltrametti and Sommese \cite{BS} and Ein and Fulton \cite{F2} for the idea of the proof.

The proof generalizes with small modification to an arbitrary field. A part that requires a little care over a nonclosed field is the proof of the ``Generalized inequality of Hodge type'' (Theorem 1.6.1 \cite{L}). We may assume that $k$ is an infinite field, by making a base change by a rational function field $k(t)$ if necessary.
As in the proof in \cite{L}, we reduce to the case where $\delta_1,\ldots,\delta_d$ are ample on a projective variety, and establish the
theorem by induction on $d$. For the case $d=2$, we resolve the singularities of $X$ (\cite{Li} or \cite{CJS}) and then apply the Hodge index theorem (Theorem  1.9 \cite{H} or Theorem B.27 \cite{K2}). Theorem 1.9 \cite{H} is proven with the assumption that $k$ is algebraically closed, but the proof is valid over an arbitrary field, using Lemma B.28 \cite{K2} instead of Corollary 1.8 \cite{H}.
The assumption that $S$ is geometrically irreducible in the statement of Theorem B.27 \cite{K2} is not necessary. The proof is valid without extending the ground field to the algebraic closure.

Finally, to reach the equality (1.24) of \cite{L}, and to acheive variant 1.6.2 \cite{L}, we must invoke a Bertini theorem which is valid over an infinite field, Theorem \ref{Bertini}.

\end{proof}

\begin{Remark} The conclusions of Theorem \ref{KT1} are true for nef line bundles on an irreducible (but possibly not reduced) proper scheme $X$ over a field $k$ (as is proven in \cite{L} when $k$ is algebraically closed). There exists $r\in \ZZ_{>0}$ such that $X=rX_{\rm red}$ (as a cycle). 
$$
(\alpha_1\cdot\ldots\cdot\alpha_d)=(\alpha_1\cdot\ldots\cdot\alpha_d\cdot X)=r(\alpha_1\cdot\ldots\cdot\alpha_d\cdot X_{\rm red})
$$
so the inequality holds from the integral case (and the fact that $(r^{\frac{1}{p}})^{p}=r$).

However, the inequality in Theorem \ref{KT1} fails if $X$ is not integral, even for ample divisors. The following is a simple example.
Let $X$ be the disjoint union of $X_1$ and $X_2$ where each $X_i$ is isomorphic to $\PP^2$. Define line bundles $\alpha_1$ and $\alpha_2$ on $X$ by 
$$
\alpha_1|X_1=\mathcal O_{X_1}(n)\mbox{ and }\alpha_1|X_2=\mathcal O_{X_2}(1),
$$
$$
\alpha_2|X_1=\mathcal O_{X_1}(1)\mbox{ and }\alpha_2|X_2=\mathcal O_{X_2}(n).
$$
$\alpha_1$ and $\alpha_2$ are both ample on $X$. Since $(\alpha_1\cdot\alpha_2)=2n$ and $(\alpha_1^2)=(\alpha_2^2)=n^2+1$,
the inequality fails for $n\ge 2$.

Using the method of Example 5.5 \cite{C1}, we can find a connected (but not integral) example where the inequality fails, essentially by
joining $X_1$ and $X_2$ at a point.
\end{Remark}

As a corollary, we obtain (Teissier, \cite{T}, \cite{T2}, and  Example 1.6.4 \cite{L})

\begin{Corollary}\label{KT2}(Khovanskii  Teissier inequalities)
Suppose that $\alpha,\beta\in N^1(X)$ are nef on $X$. Let $s_i=(\alpha^i\cdot\beta^{d-i})$. 
Then 
\begin{equation}\label{eq100}
s_i^2\ge s_{i-1}s_{i+1}
\end{equation}
for $1\le i\le d-1$,
\begin{equation}\label{eq101}
s_i^d\ge s_0^{d-i}s_d^i
\end{equation}
for $0\le i\le d$, and 
\begin{equation}\label{eq104}
((\alpha+\beta)^d)^{\frac{1}{d}}\ge (\alpha^d)^{\frac{1}{d}}+(\beta^d)^{\frac{1}{d}}.
\end{equation}
\end{Corollary}

\begin{proof} 
to obtain (\ref{eq100}, Apply (\ref{eq102}) with $p=2$, $\alpha_1=\alpha$, $\alpha_2=\beta$, $\alpha_j=\alpha$ for $3\le j\le i+1$ 
and $\alpha_j=\beta$ for $i+2\le j\le d$. To obtain (\ref{eq101}), apply (\ref{eq103}) with 
$$
\alpha_1=\cdots=\alpha_i=\alpha\mbox{ and }\alpha_{i+1}=\cdots=\alpha_d=\beta.
$$
Finally, to obtain (\ref{eq104}), expand
\begin{equation}\label{eq105}
((\alpha+\beta)^d)=\sum_{i=0}^d\binom{d}{i}s_i\ge \sum_{i=0}^d\binom{d}{i}(\alpha^d)^{\frac{i}{d}}(\beta^d)^{\frac{d-i}{d}}
=((\alpha^d)^{\frac{1}{d}}+(\beta^d)^{\frac{1}{d}})^d.
\end{equation}

\end{proof}

The corollary tells us that the sequence $\log s_0, \log s_1,\ldots, \log s_d$ is  concave; that is,
$$
\log s_i\ge \frac{1}{2}(\log s_{i-1}+\log s_{i+1})
$$
for $1\le i\le d-1$.

The sequence $\log s_0,\ldots,\log s_d$ is affine if there exist constants $a$ and $b$ such that
$$
\log s_i=ai+b
$$
for $0\le i\le d$. This condition holds if and only if
$$
\log s_i= \frac{1}{2}(\log s_{i-1}+\log s_{i+1})
$$
for $1\le i\le d-1$.

\begin{Lemma}\label{Lemma62} Suppose that $\mathcal L_1,\ldots,\mathcal L_d\in N^1(X)$ are nef and big. Then
$(\mathcal L_1\cdot\ldots\cdot\mathcal L_d)>0$. 
\end{Lemma}

\begin{proof} Since $\mbox{Nef}(X)$ is the closure of $\mbox{Amp}(X)$, we have that $(\mathcal M_1\cdot\ldots\cdot\mathcal M_p\cdot V)\ge 0$
whenever $\mathcal M_1,\ldots,\mathcal M_p$ are nef and $V$ is a closed $p$-dimensional subvariety of $X$. 

Let $H$ be a very ample line bundle on $X$. Since the $\mathcal L_i$ are big, there are $\epsilon>0$ and effective classes $E_i$
such that $\mathcal L_i=\epsilon H+E_i$ for $1\le i\le d$. We then compute using the multilinearity of the intersection product and Propositions I.2.4 and I.2.5 \cite{K} that
$$
(\mathcal L_1\cdot\ldots\cdot\mathcal L_d)\ge \epsilon^d(H^d)>0.
$$
\end{proof}

We remark that if $\alpha\in N^1(X)$ is nef, then $\alpha$ is big if and only if $(\alpha^d)>0$.

\begin{Proposition}\label{Prop63} Suppose that $ \alpha,\beta\in N^1(X)$ are nef with $(\alpha^d)>0, (\beta^d)>0$. Then  the following are equivalent.
\begin{enumerate}
\item[1)] $s_i^2=s_{i-1}s_{i+1}$ for $1\le i\le d$
\item[2)] $s_i^d=s_0^{d-i}s_d^i$ for $0\le i\le d$
\item[3)] $s_{d-1}^d=s_0s_d^{d-1}$ 
\item[4)] $(\alpha+\beta)^d)^{\frac{1}{d}}=(\alpha^d)^{\frac{1}{d}}+(\beta^d)^{\frac{1}{d}}$.
\end{enumerate}
\end{Proposition}

\begin{proof} All $s_i>0$ by Lemma \ref{Lemma62}. We first establish the equivalence of 1) and 3). From (\ref{eq100}), we obtain
$$
\frac{s_{d-1}}{s_0}=(\frac{s_{d-1}}{s_{d-2}})(\frac{s_{d-2}}{s_{d-3}})\cdots(\frac{s_1}{s_0})\ge (\frac{s_d}{s_{d-1}})^{d-1}.
$$
We have equality of the left and right hand terms if and only if $s_{d-1}^d=s_0s_d^{d-1}$ and if and only if all of the
inequalities of (\ref{eq100}) are equalities.

From the inequalities
$$
(\frac{s_i}{s_{i-1}})^{d-i}\cdots(\frac{s_1}{s_0})^{d-i}\ge 
(\frac{s_d}{s_{d-1}})^i\cdots(\frac{s_{i+1}}{s_i})^i
$$
we obtain the equivalence of 2) and 1). We obtain the equivalence of 2) and 4) from (\ref{eq105}).
\end{proof}

\begin{Theorem}\label{KT} Suppose that $\alpha_1,\ldots,\alpha_d\in \mbox{Psef}(\mathcal X)$. Then for every $1\le p\le d$, we have
$$
<\alpha_1\cdot\ldots\cdot \alpha_d>\ge <\alpha_1^p\cdot\alpha_{p+1}\cdot\ldots\cdot\alpha_d>^{\frac{1}{p}}\ldots
<\alpha_p^p\cdot\alpha_{p+1}\cdot\ldots\cdot\alpha_d>^{\frac{1}{p}}.
$$
In particular,
$$
<\alpha_1\cdot\ldots\cdot\alpha_d>\ge <\alpha_1^d>^{\frac{1}{d}}\ldots<\alpha_d^d>^{\frac{1}{d}}.
$$
\end{Theorem}

\begin{proof} 
First suppose that $\alpha_1,\ldots,\alpha_d$ are big. $L^0(\mathcal X)=\RR$ has the Euclidean topology. Let
$$
S=\left\{
(\beta_1\cdot\ldots\cdot\beta_d)\mid \begin{array}{l}\mbox{  $\beta_i\in M^1(\mathcal X)$ is nef for $1\le i\le d$}\\
\mbox{ and $D_i=\beta_i-\alpha_i$ is an effective $\QQ$-Cartier divisor}
\end{array}\right\}
$$ 
and
$$
S_i=\left\{
(\beta_i^p\cdot\beta_{p+1}\cdot\ldots\cdot\beta_d))\mid 
\begin{array}{l}\mbox{  $\beta\in M^1(\mathcal X)$ is nef 
 and $D_j=\beta_j-\alpha_j$ }\\
 \mbox{ is an effective $\QQ$-Cartier divisor for $j\in \{i,p+1,\ldots,d\}$}
\end{array}\right\}
$$ 
for $1\le i\le p$. The sets $S$ and $S_i$ are nets and $<\alpha_1\cdot\ldots\cdot\alpha_d>$, $<\alpha_i^p\cdot\alpha_{p+1}\cdot\ldots\cdot\alpha_d>$ are their respective limit points by Proposition \ref{Prop30}. Thus given $\epsilon>0$, there exist $\beta_1,\ldots,\beta_d\in \mbox{Nef}(\mathcal X)$ such that 
$$
|<\alpha_1\cdot\ldots\cdot\alpha_d>-(\beta_1\cdot\ldots\cdot\beta_d)|<\epsilon
$$
and
$$
|<\alpha_i^p\cdot\alpha_{p+1}\cdot\ldots\cdot\alpha_d>-(\beta_i^p\cdot\beta_{p+1}\cdot\ldots\cdot\beta_d))>|<\epsilon
$$
for $1\le i\le p$ by Lemma \ref{Lemma11}. We have that
$$
(\beta_1\cdot\ldots\cdot\beta_d)\ge (\beta_1^p\cdot\beta_{p+1}\cdot\ldots\cdot\beta_d)^{\frac{1}{p}}\cdots(\beta_p^p\cdot\beta_{p+1}\cdot\ldots\beta_d)^{\frac{1}{p}}
$$
by Theorem \ref{KT1}. Letting $\epsilon$ go to zero, we obtain the conclusions of the theorem.

Finally, the case when $\alpha_1,\ldots,\alpha_d$ are pseudoeffective follows from the big case, Definition \ref{Def2},  and continuity
 of the function $\RR_{>0}\rightarrow L^0(\mathcal X)=\RR$, defined by
 $$
 \begin{array}{lll}
 t&\mapsto& <(\alpha_1+t\omega)\cdot\ldots\cdot(\alpha_d+t\omega)>\\
 &&-
 <(\alpha_1+t\omega)^p\cdot(\alpha_{p+1}+t\omega)\cdot\ldots\cdot(\alpha_d+t\omega)>^{\frac{1}{p}}\cdots\\
 &&\cdots <(\alpha_p+t\omega)^p\cdot(\alpha_{p+1}+t\omega)\cdot\ldots\cdot(\alpha_d+t\omega)>^{\frac{1}{p}}
 \end{array}
 $$
 where $\omega$ is a fixed big class.
\end{proof}

\begin{Theorem}\label{Theorem50} Suppose that $\alpha_{p+1},\ldots,\alpha_d\in \mbox{Psef}(\mathcal X)$. Then the function 
$$
\alpha\mapsto <\alpha^p\cdot\alpha_{p+1}\cdot\ldots\cdot\alpha_d>^{\frac{1}{p}}
$$
is homogeneous of degree 1 and concave on $\mbox{Psef}(\mathcal X)$. In particular, the function 
$$
\alpha\mapsto <\alpha^d>^{\frac{1}{d}}
$$
 has this property.
\end{Theorem} 

\begin{proof} Homogeneity follows from Lemma \ref{LemmaSym}. We will establish that the function is concave. Suppose that $\alpha,\beta\in \mbox{Psef}(\mathcal X)$ and $0\le t\le 1$.
We have that
$$
\begin{array}{l}
<(t\alpha+(1-t)\beta)^p\cdot\alpha_{p+1}\cdot\ldots\cdot\alpha_d>\\
\ge\sum_{i=0}^p\binom{p}{i}<(t\alpha)^i\cdot((1-t)\beta)^{p-i}\cdot\alpha_{p+1}\cdot\ldots\cdot\alpha_d>\\
\ge \sum_{i=0}^p\binom{p}{i}<(t\alpha)^p\cdot\alpha_{p+1}\cdot\ldots\cdot\alpha_d>^{\frac{i}{p}}
<((1-t)\beta)^p\cdot\alpha_{p+1}\cdot\ldots\cdot\alpha_d>^{\frac{p-i}{p}}\\
= (t<\alpha^p\cdot\alpha_{p+1}\cdot\ldots\cdot\alpha_d>^{\frac{1}{p}}+(1-t)<\beta^p\cdot\alpha_{p+1}\cdot\ldots\cdot \alpha_d>^{\frac{1}{p}})^{p}.
\end{array}
$$
The second line follows from super additivity in Lemma \ref{LemmaSym}, and the third line follows from Theorem \ref{KT}.
\end{proof}

Suppose that $\alpha,\beta\in N^1(X)$ are big classes. There exists $t>0$ such that 
$\alpha-t\beta$ is big. When $X$ is projective, this follows from (\ref{eq80}). Otherwise, let $f:Y\rightarrow X$ be a birational morphism where $Y$ is a nonsingular projective variety. There exists $t>0$ such that $f^*(\alpha)-tf^*(\beta)$ is big. Thus
$\alpha-t\beta$ is big by (\ref{eq92}).

\begin{Definition} The slope of $\beta$ with respect to $\alpha$ is
$$
s=s(\alpha,\beta)=\sup\{t>0\mid \alpha\ge t\beta\}.
$$
\end{Definition}
Since $\mbox{Psef}(X)$ is closed and $\mbox{Big}(X)$ is open, we have that $\alpha\ge s\beta$ and $\alpha-t\beta$ is big for $t<s$. We have that $\alpha=\beta$ if and only if $s(\alpha,\beta)=s(\beta,\alpha)=1$.

Theorem \ref{TheoremF} is proven in Theorem F of \cite{BFJ} when $k$ is algebraically closed of characteristic zero.

\begin{Theorem}\label{TheoremF}(Diskant's inequality) Suppose that $X$ is a complete variety of dimension $d$ over a field $k$, $\alpha,\beta\in N^1(X)$ are  nef with  $(\alpha^d)>0$, $(\beta^d)>0$ and $s=s(\alpha,\beta)$ is the slope of $\beta$ with respect to $\alpha$. 
Then
$$
(\alpha^{d-1}\cdot\beta)^{\frac{d}{d-1}}-(\alpha^d)(\beta^d)^{\frac{1}{d-1}}\ge ((\alpha^{d-1}\cdot\beta)^{\frac{1}{d-1}}-s(\beta^d)^{\frac{1}{d-1}})^d.
$$
\end{Theorem}

\begin{proof} Let $\alpha_t=\alpha-t\beta$ for $t\ge 0$. By the definition of the slope $s=s(\alpha,\beta)$, and since 
bigness is an open condition,
 we have that $\alpha_t$ is big if and only if $t<s$. By Theorem \ref{TheoremA}, $f(t)={\rm vol}(\alpha_t)$ is differentiable for $t<s$, with $f'(t)=-d<\alpha_t^{d-1}>(\beta)$. We have that $f(0)=(\alpha^d)$ and $f(t)\rightarrow  0$ as $t\rightarrow s$ by continuity of ${\rm vol}$, so we have that
$$
(\alpha^d)=d\int_{t=0}^s<\alpha_t^{d-1}>(\beta) dt.
$$

From concavity in Theorem \ref{Theorem50}, we have the following formula. Suppose that $\overline\alpha$ and $\overline\beta$ are in 
$\mbox{Psef}(\mathcal X)$ and $0\le u\le 1$. Then
\begin{equation}\label{eq70}
<(u\overline\alpha+(1-u)\overline\beta)^{d-1}\cdot\beta>^{\frac{1}{d-1}}\ge
u<\overline\alpha^{d-1}\cdot\beta>^{\frac{1}{d-1}}+(1-u)<\overline\beta^{d-1}\cdot\beta>^{\frac{1}{d-1}}.
\end{equation}

For $0\le t\le s$, set 
$$
u=\frac{t}{2s},\,\,\overline\alpha=2s\beta\mbox{ and }\overline\beta=\frac{1}{(1-\frac{t}{2s})}\alpha_t.
$$
Substituting into (\ref{eq70}), we obtain
$$
<\alpha^{d-1}\cdot\beta>^{\frac{1}{d-1}}\ge t<\beta^d>^{\frac{1}{d-1}}+<\alpha_t^{d-1}\cdot\beta>^{\frac{1}{d-1}}.
$$
Let $S$ be the set of Proposition \ref{Prop30} used to compute $<\alpha_t^{d-1}>$ and $S'$ be the set of Proposition \ref{Prop30}
used to compute $<\alpha_t^{d-1}\cdot\beta>$. Since $\beta$ is nef, 
$$
S'=\{z\cdot\beta\mid z\in S\}.
$$
By Lemma \ref{Lemma11},
$$
<\alpha_t^{d-1}>(\beta)=<\alpha_t^{d-1}\cdot\beta>.
$$
We thus have
$$
(<\alpha_t^{d-1}>(\beta))^{\frac{1}{d-1}}+t(\beta^d)^{\frac{1}{d-1}}\le(\alpha^{d-1}\cdot\beta)^{\frac{1}{d-1}},
$$
by Proposition \ref{Prop11}. We  obtain
$$
(\alpha^d)\le d\int_{t=0}^s(\alpha^{d-1}\cdot\beta)^{\frac{1}{d-1}}-t(\beta^d)^{\frac{1}{d-1}})^{d-1}dt,
$$
and the result follows since
$$
\frac{d}{dt}((\alpha^{d-1}\cdot\beta)^{\frac{1}{d-1}}-t(\beta^d)^{\frac{1}{d-1}})^d=-d(\beta^d)^{\frac{1}{d-1}}((\alpha^{d-1}\cdot\beta)^{\frac{1}{d-1}}-t(\beta^d)^{\frac{1}{d-1}})^{d-1}.
$$
\end{proof}

In \cite{T}, Teissier defines the inradius of $\alpha$ with respect to $\beta$ as
$$
r(\alpha;\beta)=s(\alpha,\beta)
$$
and the outradius of $\alpha$ with respect to $\beta$ as
$$
R(\alpha;\beta)=\frac{1}{s(\beta,\alpha)}.
$$
As pointed out  in \cite{BFJ}, The Diskant inequality \ref{TheoremF} answers  Problem B in \cite{T}. In the following theorem, we
write out explicitly the bounds asked for in Problem B \cite{T}.

\begin{Theorem}\label{TheoremG} Suppose that $\alpha,\beta\in N^1(X)$ are nef with  $(\alpha^d)>0$, $(\beta^d)>0$  on a complete variety $X$ of dimension $d$ over a field $k$.
Then
\begin{equation}\label{eq106}
\frac{s_{d-1}^{\frac{1}{d-1}}-(s_{d-1}^{\frac{d}{d-1}}-s_0^{\frac{1}{d-1}}s_d)^{\frac{1}{d}}}{s_0^{\frac{1}{d-1}}}
\le r(\alpha;\beta)\le \frac{s_d}{s_{d-1}}
\end{equation}
\end{Theorem}

\begin{proof} Let $s=s(\alpha,\beta)=r(\alpha,\beta)$. Since $\alpha\ge s\beta$ and $\alpha$, $\beta$ are nef, we have that $(\alpha^d)\ge s(\beta\cdot\alpha^{d-1})$ by Proposition \ref{PropNef}. This gives us the upper bound. We also have that
\begin{equation}\label{eq110} 
(\alpha^{d-1}\cdot\beta)^{\frac{1}{d-1}}-s(\beta^d)^{\frac{1}{d-1}}\ge 0.
\end{equation}
 We  obtain the lower bound from Theorem \ref{TheoremF} (using (\ref{eq110})
and the inequality $s_{d-1}^d\ge s_0s_d^{d-1}$ to ensure that the bound is a positive real number).
\end{proof}

\begin{Theorem}\label{TheoremH} Suppose that $\alpha,\beta\in N^1(X)$ are nef with $(\alpha^d)>0$, $(\beta^d)>0$  on a complete variety $X$ of dimension $d$ over a field $k$.
Then
\begin{equation}\label{eq107}
\frac{s_{d-1}^{\frac{1}{d-1}}-(s_{d-1}^{\frac{d}{d-1}}-s_0^{\frac{1}{d-1}}s_d)^{\frac{1}{d}}}{s_0^{\frac{1}{d-1}}}
\le r(\alpha;\beta)\le \frac{s_d}{s_{d-1}}\le\frac{s_1}{s_0}\le R(\alpha;\beta)\le 
\frac{s_d^{\frac{1}{d-1}}}{s_1^{\frac{1}{d-1}}-(s_1^{\frac{1}{d-1}}-s_d^{\frac{1}{d-1}}s_0)^{\frac{1}{d}}}
\end{equation}
\end{Theorem}

\begin{proof} By Theorem \ref{TheoremG}, we have that
$$
\frac{s_1^{\frac{1}{d-1}}-(s_1^{\frac{d}{d-1}}-s_d^{\frac{1}{d-1}}s_0)^{\frac{1}{d}}}{s_d^{\frac{1}{d-1}}}
\le s(\beta,\alpha)\le\frac{s_0}{s_1}.
$$
The theorem now follows from the fact that $R(\alpha,\beta)=\frac{1}{s(\beta,\alpha)}$, (\ref{eq100}) and Theorem \ref{TheoremG}.
\end{proof}

The bounds of  Theorems \ref{TheoremG} and \ref{TheoremH}   are exactly those obtained by Teissier in Proposition 3.2 \cite{T} when $X$ is an integral projective surface 
 and $\alpha,\beta$ are line bundles on $X$ with $(\alpha^2)>0$, $(\beta^2)>0$ and with $\alpha+\beta$ is ample.  His $s_i$ is our $s_{d-i}$.

We immediately obtain the following corollary.

\begin{Corollary}\label{Bonnessen}(Bonnesen's inequality)
 Suppose that $X$ is a complete surface over a field $k$, and that $\alpha,\beta\in N^1(X)$ are nef with $(\alpha^2)>0$, $(\beta^2)>0$. Then
 $$
 \frac{s_0^2}{4}(R(\alpha;\beta)-r(\alpha;\beta))^2\le s_1^2-s_0s_2
 $$
 \end{Corollary}
 
Corollary \ref{Bonnessen} is obtained by Teissier in  \cite{T} when 
$X$ is an integral projective surface and $\alpha,\beta$ are line bundles on $X$ with $(\alpha^2)>0$, $(\beta^2)>0$ and with $\alpha+\beta$ is ample. Again, 
 his $s_i$ is our $s_{d-i}$.
 
 The following theorem is proven in Theorem D of \cite{BFJ} when $k$ is an algebraically closed field of characteristic zero.

\begin{Theorem}\label{TheoremD} Suppose that $\alpha,\beta\in N^1(X)$ are nef with $(\alpha^d)>0$, $(\beta^d)>0$  on a complete variety $X$ of dimension $d$ over a field $k$. Then the following are equivalent:
\begin{enumerate}
\item[i)] $\alpha$ and $\beta$ satisfy the equivalent conditions of Proposition \ref{Prop63}
\item[ii)]  $\alpha$ and $\beta$ are proportional in $N^1(X)$.
\end{enumerate}
\end{Theorem}

\begin{proof}
Suppose that $\alpha$ and $\beta$ satisfy the equivalent conditions of Proposition \ref{Prop63}. Since the intersection products are homogeneous, we can assume that $(\alpha^d)=(\beta^d)=1$. Then $s_i=1$ for all $i$. Theorem \ref{TheoremH} then implies that
$r(\alpha;\beta)=R(\alpha;\beta)=1$. Thus $s(\alpha,\beta)=s(\beta,\alpha)=1$ and so $\alpha=\beta$.
\end{proof}

 Teissier observes in \cite{T} that Theorem \ref{TheoremD} is proven in some cases for surfaces in Expose XIII of \cite{SGA6}.

\begin{Remark}
The conclusions of Theorem \ref{TheoremD} do not hold if $\alpha$ and $\beta$ are only nef and not big. A simple example is obtained by letting $W$ be any $(d-1)$-dimensional projective variety with $\dim N^1(W)>1$ and letting $\overline\alpha$ and $\overline\beta$ be 
very ample Cartier divisors which are not proportional in $N^1(W)$. Let $X=W\times_k\PP^1$ with projection $\pi:X\rightarrow W$.
Let $\alpha=\pi^*(\overline\alpha)$ and $\beta=\pi^*(\overline\beta)$. Then
$$
s_i=(\alpha^i\cdot\beta^{d-i})=0
$$
for all $i$ (by Propositions I.2.1 and I.2.6 \cite{K}), but $\alpha$ and $\beta$ are not proportional.

Theorem \ref{TheoremD} does hold if $X$ is an integral (possibly not reduced) proper scheme over a field $k$ and $\alpha|X_{\rm red}$
and $\beta|X_{\rm red}$ are big, since the theorem holds on the variety $D_{\rm red}$, and $X$ is a multiple of $X_{\rm red}$ as a cycle.
\end{Remark}

\end{document}